\documentclass[reqno]{amsart}
\usepackage{amsfonts,amsmath,amsthm,amssymb,mathrsfs}
\usepackage{color}
\usepackage{amsmath,amssymb,amsfonts,bm}
\usepackage{graphicx}
\usepackage{newunicodechar}
\usepackage{xcolor}
\numberwithin{equation}{section}
\usepackage[pdfstartview=FitH,colorlinks=true]{hyperref}
\hypersetup{urlcolor=red, linkcolor=blue,citecolor=red}
\allowdisplaybreaks

\theoremstyle{plain}
\newtheorem{theorem}{Theorem}[section]
\newtheorem{remark}[theorem]{Remark}
\newtheorem{lemma}[theorem]{Lemma}
\newtheorem{proposition}[theorem]{Proposition}

\newtheorem{corollary}[theorem]{Corollary}
\begin{document}

\title[Landau equation]
{Analytic Gelfand-Shilov smoothing effect of\\
the spatially homogeneous Landau equation
}

\author[Hao-Guang Li \& Chao-Jiang Xu]
{Hao-Guang Li and Chao-Jiang Xu}

\address{Hao-Guang Li,
\newline\indent
School of Mathematics and Statistics, South-Central Minzu University,
\newline\indent
430074, Wuhan, P. R. China}
\email{lihaoguang@mail.scuec.edu.cn}

\address{Chao-Jiang Xu,
\newline\indent
College of Mathematics and Key Laboratory of Mathematical MIIT,
\newline\indent
Nanjing University of Aeronautics and Astronautics, Nanjing 210016, China
\newline\indent
Universit\'e de Rouen-Normandie, CNRS UMR 6085, Laboratoire de Math\'ematiques
\newline\indent
76801 Saint-Etienne du Rouvray, France
}
\email{xuchaojiang@nuaa.edu.cn}

\date{\today}



\begin{abstract}
In this work, we study the nonlinear spatially homogeneous Landau equation with hard potential in a close-to-equilibrium framework, we show that the solution to the Cauchy problem with $L^2$ initial datum enjoys a analytic Gelfand-Shilov regularizing effect in the class $S^1_1(\mathbb{R}^3)$, meaning that the solution of the Cauchy problem and its Fourier transformation are analytic for any positive time, the evolution of analytic radius is similar to the heat equation.
\end{abstract}

\maketitle
%\tableofcontents

%%%%%%%%%%%%%%%%%%%%%%%%%%%%%%
%%%%%%%%%%%%%%%%%%%%%%%%%%%%%%

\section{Introduction}\label{S1}
In this work, we study the spatially homogeneous Landau equation
\begin{equation}\label{eq1.10}
\left\{
\begin{array}{ll}
   \partial_t f= Q(f,f),\\
  f|_{t=0}=f_0,
\end{array}
\right.
\end{equation}
where $f=f(t,v)\ge0$ is the density distribution function depending on the variables
$v\in\mathbb{R}^{3}$ and the time $t\geq0$. The Landau bilinear collision operator is given by
\begin{equation}\label{Landau-operaor}
Q(g,f)(v)
=\triangledown_v\cdot\left(\int_{\mathbb{R}^{3}}
\phi(v-v_*)
\big(g(v_*)(\triangledown_vf)(v)-(\triangledown_vg)(v_*)f(v)\big)
dv_*\right),
\end{equation}
where $\phi(v)=(\phi^{ij}(v))_{1\leq\,i,j\leq3}$ stands for the non-negative symmetric matrix
$$
\phi(v)=(|v|^2\textbf{I}_3-v\otimes\,v)|v|^{\gamma}\in\,M_3(\mathbb{R}),\quad\,\gamma\geq-3.
$$
We shall study the fluctuation of the Landau equation \eqref{eq1.10} near the absolute Maxwellian distribution
$$
\mu(v)=(2\pi)^{-\frac 32}e^{-\frac{|v|^{2}}{2}}.
$$
Considering the perturbation of density distribution function
$$
f(t,v)=\mu(v)+\sqrt{\mu}(v)g(t,v),
$$
since $Q(\mu,\mu)=0$, the Cauchy problem \eqref{eq1.10} is reduced to the Cauchy problem
\begin{equation} \label{eq-1}
\left\{ \begin{aligned}
         &\partial_t g+\mathcal{L}g={\bf \Gamma}(g, g),\,\,\, t>0,\, v\in\mathbb{R}^3,\\
         &g|_{t=0}=g_{0},
\end{aligned} \right.
\end{equation}
with $g_0(v)=\mu^{-\frac 12}f_0(v) -\sqrt{\mu}$, where
$$
{\bf \Gamma}(g, g)=\mu^{-\frac 12}Q(\sqrt{\mu}g,\sqrt{\mu}g),\quad
\mathcal{L}g=\mathcal{L}_1g+\mathcal{L}_2g,
$$
and
\begin{equation}\label{K-A}
\mathcal{L}_1g=-\Gamma(\sqrt{\mu},g), \quad \mathcal{L}_2g =-\Gamma(g,\sqrt{\mu})
\end{equation}

We introduce the following Gelfand-Shilov spaces $S^{\mu}_{\nu}(\mathbb{R}^3)$, with $\mu,\,\nu>0,$\,$\mu+\nu\geq1,$\, which is the smooth functions $f\in\,C^{+\infty}(\mathbb{R}^3)$ satisfying:
\begin{equation*}%\label{1.7+0}
\exists\, B>0,\, \,
\sup_{v\in\mathbb{R}^3}|v^{\beta}\partial^{\alpha}_vf(v)|\leq\,B^{|\alpha|+
|\beta|+1}(\alpha!)^{\mu}(\beta!)^{\nu},\,\,\forall\,\alpha,\,\beta\in\mathbb{N}^3.
\end{equation*}
This Gelfand-Shilov space can be characterized as the
subspace of Schwartz functions $f\in\,\mathscr{S}(\mathbb{R}^3)$ such that,
$$
\exists\, C>0,\,c_0>0,\,|f(v)|\leq Ce^{-c_0|v|^{\frac{1}{\nu}}},\,\,v\in\mathbb{R}^3\,\,\text{and}\,\,|\hat{f}(\xi)|\leq Ce^{-c_0|\xi|^{\frac{1}{\mu}}},\,\,\xi\in\mathbb{R}^3,
$$
where $c_0$ is called the Gelfand-Shilov radius. $S^{1}_{1}(\mathbb{R}^3)$ is called analytic Gelfand-Shilov space.

\bigskip
The existence, uniqueness of the solution to Cauchy problem for the spatially homogeneous Landau equation has already been treated in \cite{DV},\cite{Villani1998-2} under rather weak assumption on the initial datum. Moreover, in the hard potential case, they prove the smoothness of the solution in
$C^\infty(]0, +\infty[; \mathcal{S}(\mathbb{R}^3))$.  In \cite{ChenLiXu6}, Chen-Li-Xu improve this smoothing property and prove that the solution is in fact analytic for any $t>0$ (See  \cite{ChenLiXu3,ChenLiXu5} for the Gevrey regularity). For the analytic smoothing effect, we can also refer to \cite{LX4} and \cite{MX}.

In the Maxwellian molecules case, in \cite{NYKC2}, Lerner,  Morimoto, Pravda-Starov and Xu study the spatially homogeneous non-cutoff Boltzmann equation and Landau equation in a close-to-equilibrium framework and show that the solution enjoys the Gelfand-Shilov smoothing effect (see also \cite{Li2}, \cite{LX-3} and \cite{MPX}). This implies that the nonlinear spatial homogeneous Landau equation has the same smoothing effect properties as the classic heat equation or harmonic oscillators heat equation.  In addition,  starting from a $L^2$ initial datum at $t=0$, the solution of Cauchy problem is spatial analytic for any $t>0$ and the analytic radius is $c_0 t^{\frac 12}$.  In the non-Maxwellian case, we can't use the Fourier transformation and spectral decomposition as in \cite{NYKC2,Li2,LX-3,MPX}.

Now we define the creation and annulation operators,
\begin{equation}\label{H3}
A_{\pm,k}=\frac{v_{k}}{2}\mp\partial_k,\quad 1\leq k\leq 3,
\end{equation}
and
$$
A^{\alpha}_{+}=A^{\alpha_1}_{+,1}A^{\alpha_2}_{+,2}A^{\alpha_3}_{+,3},\quad A^{\alpha}_{-}=A^{\alpha_1}_{-,1}A^{\alpha_2}_{-,2}A^{\alpha_3}_{-,3},\quad \alpha\in\mathbb{N}^3.
$$
Moreover, we  define that the gradient of $\mathcal{H}$ as follows
\begin{equation}\label{gradientH}
\nabla_{\mathcal{H}_+}=(A_{+,1}, A_{+,2}, A_{+,3}),\quad \nabla_{\mathcal{H}_-}=(A_{-,1}, A_{-,2}, A_{-,3})
\end{equation}
and then define the norm, for $m\ge1$,
\begin{equation}\label{multi-indices}
\|\nabla_{\mathcal{H}_+}^{m}u\|^2_{L^{2}(\mathbb{R}^{3})}=\sum^3_{k=1}\|A_{+,k}\nabla_{\mathcal{H}_+}^{m-1}u\|^2_{L^{2}(\mathbb{R}^{3})}=\sum_{|\alpha|=m}\frac{m!}{\alpha!}\|A^{\alpha}_{+}u\|^2_{L^{2}(\mathbb{R}^{3})}.
\end{equation}
Where for the multi-indices, we use the notation from the binomial expansion
\begin{align*}
|\xi|^{2m}=(\xi_1^2+\xi_2^2+\xi_3^2)^m=\sum_{|\alpha|=m}\frac{m!}{\alpha!}\xi_1^{2\alpha_1}\xi_2^{2\alpha_2}\xi_3^{2\alpha_3}.
\end{align*}
The main theorem of this paper is the following analytic Gelfand-Shilov smoothing effect of a smooth solution of the Cauchy problem \eqref{eq-1}.

\begin{theorem}\label{trick}
Let $g$ be a smooth solution of the Cauchy problem \eqref{eq-1} with $\gamma\ge 0$,  and there exist a positive constant $\epsilon_0>0$ small enough, such that
\begin{equation}\label{1.7+01}
\|g\|_{L^{\infty}([0, +\infty[; \, L^2(\mathbb{R}^3))}\leq \epsilon_0.
\end{equation}
Then, there exists $C>0$ such that for any $m\in \mathbb{N}$,
we have
\begin{equation}\label{1.7+1}
\|\tilde{t}^{\frac{m}{2}}\nabla_{\mathcal{H}_+}^{m}g\|_{L^{\infty}([0, +\infty[; \, L^2(\mathbb{R}^3))}\leq C^{m+1}m!,
\end{equation}
where $\tilde{t}=\min(t,1)$.
\end{theorem}

\begin{remark}
We will prove, in Appendix \ref{Appendix}, that the estimates \eqref{1.7+1} implies $g(t)\in S^1_1(\mathbb{R}^3)$ for any $t>0$, and the analytical Gelfand-Shilov radius is $c_0t^{\frac 12}$ for $0\le t\le 1$. So that we extend the results of \cite{Li2}, \cite{LX-3} and \cite{MPX} to the hard potential case,
and show that the nonlinear Cauchy problem
\eqref{eq-1} enjoys the same smoothing effect as the following Cauchy problem
\[
\left\{ \begin{aligned}
         &\partial_t f- \langle v\rangle^{\gamma}\left(\triangle_v-\frac{|v|^2}{4}\right)f=0,\,\,\, t>0,\, v\in\mathbb{R}^3,\\
         &f|_{t=0}=f_{0},
\end{aligned} \right.
\]
with $\gamma\ge 0$, we want to point out that this is a uniformly parabolic problem.
\end{remark}

This paper is arranged as follows:  In Section \ref{S2}, we introduce a new expansion of the linear and nonlinear Landau operators.  Then we give the spectral analysis on the Landau operators and prove a fundamental trilinear estimate for the nonlinear Landau operator.
In Section \ref{S3},  we present a new kind of Leibniz formula. By using this formula, the trilinear estimate of the nonlinear Landau operators with gradient of $\mathcal{H}_+$ will be given.
In Section \ref{S4},  we study the coercivity for the linear Landau operator, which is crucial in the proof of Gelfand-Shilov smoothing effect for the weak solution of the Cauchy problem \eqref{eq-1}.
On the basis of the preparatory estimate, the main theorem \ref{trick} of the Gelfand-Shilov smoothing effect will be proved in Section \ref{S5}.
In the Appendix \ref{Appendix}, we introduce the Hermite operator and Gelfand-Shilov space.  Moreover, we prove in the Appendix,  the estimate \eqref{1.7+1} implies that the solution $g$ to the Cauchy problem \eqref{eq-1} enjoys the Gelfand-Shilov $S^1_1(\mathbb{R}^3)$ smoothing effect.

%%%%%%%%%%%%%%%%%%%%%%%%%%%%%%
\section{Analysis of the Landau operators}\label{S2}
In this section, we introduce the representations of linear Landau operator and nonlinear Landau operator. Then we present the preparation Lemmas for the estimate on Landau operators.

Similar to the computation of Lemma 1 in \cite{Guo-2002}, we have the following Lemma.
\subsection*{Representations of Landau operators}
\begin{lemma}\label{detail}
We have the following representations for $\mathcal{L}_1$, $\mathcal{L}_2$ and $\Gamma$:
\begin{equation}\label{Gamma}
\begin{split}
&\mathcal{L}_1g=\sum^3_{i,j=1}A_{+,i}\{(\phi^{ij}*\mu)A_{-,j}g\},\\
&\mathcal{L}_2f=-\sum^3_{i,j=1}A_{+,i}\{\sqrt{\mu}(\phi^{ij}*(\sqrt{\mu}A_{-,j} f))\},\\
&\Gamma(f,g)=\sum^3_{i,j=1}A_{+,i}\{(\phi^{ij}*(\sqrt{\mu}f))A_{+,j}g\}-\sum^3_{i,j=1}A_{+,i}\{(\phi^{ij}*
(\sqrt{\mu}A_{+,j} f))g\}\\
&\qquad\quad=\Gamma_1(f,g)+\Gamma_2(f,g).
\end{split}
\end{equation}
\end{lemma}

\begin{proof}
It is well known that $Q(\mu,\mu)=0$.  By expanding $Q(\mu+\sqrt{\mu}f,\mu+\sqrt{\mu}g)$ around $\mu$, we have
\begin{align*}
Q(\mu+\sqrt{\mu}f,\mu+\sqrt{\mu}g)&=Q(\sqrt{\mu}f,\mu)
+Q(\mu,\sqrt{\mu}g)+Q(\sqrt{\mu}f,\sqrt{\mu}g)\\
&=\sqrt{\mu}\{-\mathcal{L}_2f-\mathcal{L}_1g+\Gamma(f,g)\}.
\end{align*}
Notice that
\begin{align}
&\partial_j(\sqrt{\mu}f)=\sqrt{\mu}(\partial_j-\frac{v_j}{2})f=-\sqrt{\mu}A_{+,j}f,\label{creat-1}\\
&\sqrt{\mu}^{-1}\partial_i F=(\partial_i-\frac{v_i}{2})(\sqrt{\mu}^{-1}F)=-A_{+,i}(\sqrt{\mu}^{-1}F),\label{creat-2}
\end{align}
and for any fixed $i$ or $j$,
\[
\sum^3_{i=1}\phi^{ij}(v-v')(v_i-v'_i)=\sum^3_{j=1}\phi^{ij}(v-v')(v_j-v'_j)=0.
\]
We obtain from \eqref{Landau-operaor}, \eqref{creat-1} and \eqref{creat-2} that
\begin{align*}
\mathcal{L}_2f&=-\sqrt{\mu}^{-1}Q(\sqrt{\mu}f,\mu)\\
&=\sum^3_{i,j=1}\sqrt{\mu}^{-1}\partial_i\{v_j\mu(\phi^{ij}*(\sqrt{\mu}f))\}
+\sum^3_{i,j=1}\sqrt{\mu}^{-1}\partial_i\{\mu(\phi^{ij}*(\sqrt{\mu}A_{+,j} f))\}\\
&=-\sum^3_{i,j=1}A_{+,i}\{\sqrt{\mu}(\phi^{ij}*(v_j\sqrt{\mu}f))
-\sum^3_{i,j=1}\sqrt{\mu}(\phi^{ij}*(\sqrt{\mu}A_{+,j}f))\}\\
&=-\sum^3_{i,j=1}A_{+,i}\{\sqrt{\mu}(\phi^{ij}*(\sqrt{\mu}A_{-,j} f))\}.
\end{align*}
For $\mathcal{L}_1g$, using \eqref{creat-1} and \eqref{creat-2}, we have
\begin{align*}
\mathcal{L}_1g&=-\sqrt{\mu}^{-1}Q(\mu,\sqrt{\mu}g)\\
&=\sum^3_{i,j=1}\sqrt{\mu}^{-1}\partial_i\{(\phi^{ij}*\mu)\sqrt{\mu}A_{+,j}g\}+\sum^3_{i,j=1}\sqrt{\mu}^{-1}
\partial_i\{(\phi^{ij}*(v_j\mu ))\sqrt{\mu}g\}\\
&=-\sum^3_{i,j=1}A_{+,i}\{(\phi^{ij}*\mu)A_{+,j}g-(\phi^{ij}* \mu)v_j g\}\\
&=\sum^3_{i,j=1}A_{+,i}\{(\phi^{ij}*\mu)A_{-,j}g\}.
\end{align*}
Finally, from \eqref{Landau-operaor}, \eqref{creat-1} and \eqref{creat-2}, we have
\begin{align*}
\Gamma(f,g)
&=\sqrt{\mu}^{-1}Q(\sqrt{\mu}f,\sqrt{\mu}g)\\
&=-\sum^3_{i,j=1}\sqrt{\mu}^{-1}\partial_i\{(\phi^{ij}*(\sqrt{\mu}f))\sqrt{\mu}A_{+,j}g\}\\
&\quad+\sum^3_{i,j=1}\sqrt{\mu}^{-1}\partial_i\{(\phi^{ij}*(\sqrt{\mu}A_{+,j} f))\sqrt{\mu}g\}\\
&=\sum^3_{i,j=1}A_{+,i}\{(\phi^{ij}*(\sqrt{\mu}f))A_{+,j}g\}-\sum^3_{i,j=1}A_{+,i}\{(\phi^{ij}*(\sqrt{\mu}A_{+,j} f))g\}\\
&=\Gamma_1(f,g)+\Gamma_2(f,g).
\end{align*}
We end the calculation of Lemma \ref{detail}.
\end{proof}

\begin{remark} For $\gamma=0$,directly computation shows that
$$
\phi^{jj} * \mu=\int_{\mathbb{R}^3}\Big((v_2-w_2)^2+(v_3-w_3)^2\Big)\mu(w)dw=|v|^2-|v_j|^2-2,
$$
and for $i\neq j$,
$$
\phi^{ij}*\mu=-\int_{\mathbb{R}^3}(v_i-w_i)(v_j-w_j)\mu(w)dw=-v_iv_j.
$$
Then it follows that
\begin{align*}
\mathcal{L}_1g=2\Big(-\Delta+\frac{|v|^2}{4}-\frac{3}{2}\Big)g-\Delta_{\mathbb{S}^2}g,
\end{align*}
where
$$\Delta_{\mathbb{S}^2}=\frac{1}{2}\sum_{\substack{1\leq i,j\le3\\i\neq j}}\Big(v_i\partial_j-v_j\partial_i\Big)^2.$$
This is consistent with Proposition 2.1 in \cite{NYKC2}, we can also refer to \cite{LX-3}.
\end{remark}

\subsection*{The estimation of nonlinear Landau operators}

For the matrix $\phi$ defined in \eqref{Landau-operaor}, we denote
\begin{equation}\label{sigma}
\sigma^{ij}=\phi^{ij}*\mu,\quad
\sigma^{i}=\sum^3_{j=1}\phi^{ij}*(v_j\mu)
\end{equation}
and define, for $g\in\mathcal{S}(\mathbb{R}^3)$,
\begin{equation}\label{L2A}
\||g|\|^2_{\sigma}=\sum_{i,j=1}^{3}\int_{\mathbb{R}^3}
\Big(\sigma^{ij}\partial_ig\partial_jg+\frac{1}{4}\sigma^{ij}v_iv_jg^2\Big)dv.
\end{equation}

For any vector-valued function $G(v)=(G_{1},G_{2},G_{3})$, we define the projection to the vector $v=(v_{1},v_{2},v_{3})$ as
$$
\mathbf{P}_{v}G=\{\sum_{j=1}^{3}G_{j}v_{j}\}\frac{v}{| v|^{2}}.
$$
Recall the definition \eqref{gradientH}, we have
\begin{proposition}\label{norm} For the norm of \eqref{L2A}, we have
\begin{equation}\label{definition2}
\||g|\|^2_{\sigma}\geq C_{1}\Big(\|\langle v\rangle^{\frac{\gamma}{2}}\nabla g\|^2_{L^2(\mathbb{R}^3)}+\|\langle v\rangle^{\frac{\gamma+2}{2}}g\|^2_{L^2(\mathbb{R}^3)}\Big),
\end{equation}
and
\begin{equation}\label{definition3}
\begin{split}
 \||g|\|^2_{\sigma}\geq &C_{1}\Big(\|\langle v\rangle^{\frac{\gamma}{2}}\mathbf{P}_v\nabla_{\mathcal{H}_{\pm}}g\|^2_{L^2(\mathbb{R}^3)}+\|\langle v\rangle^{\frac{\gamma+2}{2}}(\mathbf{I}-\mathbf{P}_v)\nabla_{\mathcal{H}_{\pm}}g\|^2_{L^2(\mathbb{R}^3)}\Big)\\
\geq& C_{1}\|\langle v\rangle^{\frac{\gamma}{2}}\nabla_{\mathcal{H}_{\pm}}g\|^2_{L^2(\mathbb{R}^3)}.
\end{split}
\end{equation}
\end{proposition}
\begin{proof}
Simple calculation shows that,
\begin{align*}
\||g|\|^2_{\sigma}=\frac{1}{2}\sum_{i,j=1}^{3}\int_{\mathbb{R}^3}\sigma^{ij}\left(A_{-,i}gA_{-,j}g+A_{+,i}gA_{+,j}g\right)dv.
\end{align*}
From formula $(21)$ of Corollary 1 in \cite{Guo-2002}, there exist $C_{1}>0$, such that
\begin{equation*}
 \|| g|\|_{\sigma}^{2}\geq C_{1}\Big(\|\langle v\rangle^{\frac{\gamma}{2}}\mathbf{P}_v\nabla g\|^2_{L^2(\mathbb{R}^3)}
 +\|\langle v\rangle^{\frac{\gamma+2}{2}}(\mathbf{I}-\mathbf{P}_v)\nabla g\|^2_{L^2(\mathbb{R}^3)}+\|\langle v\rangle^{\frac{\gamma+2}{2}}g\|^2_{L^2(\mathbb{R}^3)}\Big).
\end{equation*}
Notice that $\nabla f=\mathbf{P}_{v}\nabla f+(\mathbf{I}-\mathbf{P}_{v})\nabla f$, we have then
$$
\||g|\|^2_{\sigma}\geq C_{1}\Big(\|\langle v\rangle^{\frac{\gamma}{2}}\nabla g\|^2_{L^2(\mathbb{R}^3)}+\|\langle v\rangle^{\frac{\gamma+2}{2}}g\|^2_{L^2(\mathbb{R}^3)}\Big).
$$
Moreover, from the definition of  $\mathbf{P}_{v}$, one can find that
$$\mathbf{P}_v(v_ig)=\sum_{j=1}^{3}(v_{j}g)v_{j}\frac{v_{i}}{| v|^{2}}=v_ig,$$
which means that
$$
(\mathbf{I}-\mathbf{P}_v)(A_{\pm,i}g)=\mp(\mathbf{I}-\mathbf{P}_v)(\partial_ig).
$$
Therefore, we can deduce that
\begin{equation*}
\begin{split}
\||g|\|^2_{\sigma}\geq& C_{1}\Big(\|\langle v\rangle^{\frac{\gamma}{2}}\mathbf{P}_v\nabla_{\mathcal{H}_{\pm}}g\|^2_{L^2(\mathbb{R}^3)}+\|\langle v\rangle^{\frac{\gamma+2}{2}}(\mathbf{I}-\mathbf{P}_v)\nabla_{\mathcal{H}_{\pm}}g\|^2_{L^2(\mathbb{R}^3)}\Big)\\
\geq& C_{1}\|\langle v\rangle^{\frac{\gamma}{2}}\nabla_{\mathcal{H}_{\pm}}g\|^2_{L^2(\mathbb{R}^3)}.
\end{split}
\end{equation*}
\end{proof}
We can also refer to \cite{DeL} and reference works for the spectral analysis.
For simplicity, we define in the following, for $s\in \mathbb{R}$,
$$
\|\langle v\rangle^{s}u\|_{L^{2}(\mathbb{R}^{3})}=\| u\|_{2, s},
$$
and notice that
$$
\|\nabla_{\mathcal{H}_{\pm}}^{m}u\|^2_{2,s}=\sum^3_{k=1}\|A_{+,k}\nabla_{\mathcal{H}_{\pm}}^{m-1}u\|^2_{2,s}=\sum_{|\alpha|=m}\frac{m!}{\alpha!}\|A^{\alpha}_{\pm}u\|^2_{2,s},
$$
where we use the notation $\langle v\rangle=(1+| v|^2)^{1/2}$.

\begin{lemma}\label{covo}
For any $s>-3$, we have, for $\delta>0$,
\begin{equation}\label{cov}
\int_{\mathbb{R}^3}|v-w|^{s}e^{-\delta|w|^2}dw\lesssim \langle v\rangle^{s}.
\end{equation}
\end{lemma}

The weighted of $\mathbf{P}_v\nabla_{\mathcal{H}_{\pm}}g$ and $(\mathbf{I}-\mathbf{P}_v)\nabla_{\mathcal{H}_{\pm}}g$ are different in the norm $\||\ \cdot\ |\|_{\sigma}$, we need to study the trilinear estimate of nonlinear Landau operator.

\begin{proposition}\label{3-estimate}
For $f,g,h\in\mathcal{S}(\mathbb{R}^3)$, $\gamma\ge0$, we have
\begin{equation*}%\label{estimate2}
\left|\langle\Gamma(f, g), h\rangle\right|\lesssim\| f\|_{L^2(\mathbb{R}^3)}\| |g|\|_{\sigma}\| |h|\|_{\sigma}.
\end{equation*}
\end{proposition}
\begin{proof}
In fact, the integration
\begin{align*}
\langle\Gamma(f, g), h\rangle=&\langle\Gamma_1(f, g), h\rangle+\langle\Gamma_2(f, g), h\rangle\\
=&\sum^3_{i,j=1}\iint_{\mathbb{R}^3\times \mathbb{R}^3}\phi^{ij}(v-w)\sqrt{\mu(w)}f(w)A_{+,j}g(v)
A_{-,i}h(v)dwdv\\
&-\sum^3_{i,j=1}\iint_{\mathbb{R}^3\times \mathbb{R}^3}\phi^{ij}(v-w)\sqrt{\mu(w)}A_{+,j}f(w)g(v)
A_{-,i}h(v)dwdv.
\end{align*}
For the first term $\langle\Gamma_1(f, g), h\rangle$,
we decompose the integration region $[v,w]\in \mathbb{R}^3\times \mathbb{R}^3$ into three parts:
$$
\{| v|\leq1\},\ \  \{2| w|\geq| v|,| v|\geq1\},\ \ and\ \  \{2| w|\leq| v|,| v|\geq1\}.
$$
For the first part $\{| v|\leq1\}$, since
$$|\phi^{ij}(v)|\lesssim |v|^{\gamma+2},$$
 we can deduce from the Cauchy-Schwartz's inequality, Lemma \ref{covo} and \eqref{definition3} that
\begin{align*}
\left|\langle\Gamma_1(f, g), h\rangle\right|
&\lesssim\sum^3_{i,j=1}\|f\|_{L^2(\mathbb{R}^3)}\int_{\mathbb{R}^3} \langle  v\rangle^{\gamma+2}|A_{+,j}g
A_{-,i}h|dv\\
&\leq\sum^3_{i,j=1}\|f\|_{L^2(\mathbb{R}^3)} \|\langle v\rangle^{\frac{\gamma}{2}}A_{+,j}g\|_{L^2}\|\langle v\rangle^{\frac{\gamma}{2}}A_{-,i}h\|_{L^2}\\
&\lesssim \|f\|_{L^2(\mathbb{R}^3)} \||g|\|_{\sigma}\||h|\|_{\sigma}.
\end{align*}
For the second part $\{2| w|\geq| v|,| v|\geq1\}$,  we have
$$e^{-\frac{|w|^2}{4}}\leq e^{-\frac{|w|^2}{8}} e^{-\frac{|v|^2}{32}}.$$
Similar to the proof as the first part, one can verify that
\begin{align*}
\left|\langle\Gamma_1(f, g), h\rangle\right|
&\lesssim\sum^3_{i,j=1}\|f\|_{L^2(\mathbb{R}^3)}\int_{\mathbb{R}^3} \langle v\rangle^{\gamma+2}e^{-\frac{|v|^2}{32}}|A_{+,j}g
A_{-,i}h|dv\\
&\lesssim\sum^3_{i,j=1}\|f\|_{L^2(\mathbb{R}^3)} \|\langle v\rangle^{\frac{\gamma}{2}}A_{+,j}g\|_{L^2}\|\langle v\rangle^{\frac{\gamma}{2}}A_{-,i}h\|_{L^2}\\
&\lesssim \|f\|_{L^2(\mathbb{R}^3)} \||g|\|_{\sigma}\||h|\|_{\sigma}.
\end{align*}
Now we finally consider the third part $\{2| w|\leq| v|,| v|\geq1\}.$ Expanding $\phi^{ij}(v-w)$ to get
\begin{equation}\label{expanding}
\phi^{ij}(v-w)=\phi^{ij}(v)+\sum^3_{k=1}\partial_{k}\phi^{ij}(v)w_k+\frac{1}{2}\sum_{k,l=1}^{3}\left(\int_{0}^{1}\partial_{kl}\phi^{ij}(v-sw)ds\right)w_{k}w_l.
\end{equation}
The expansion \eqref{expanding} along with the fact that
$$
\sum_{i=1}^{3}\phi^{ij}v_{i}=0, \quad \sum_{i,j=1}^{3}\partial_{k}\phi^{ij}(v)v_iv_j=-2\sum_{j=1}^{3}\phi^{kj}(v)v_j=0,
$$
show immediately
\begin{align*}
&\sum_{i,j=1}^{3}\iint_{2| w|\leq| v|,| v|\geq1}\phi^{ij}(v-w)\sqrt{\mu(w)}f(w)A_{+,j}g(v)
A_{-,i}h(v)dwdv\\
=&\sum_{i,j=1}^{3}\iint_{2| w|\leq| v|,| v|\geq1}\phi^{ij}(v)\sqrt{\mu(w)}f(w)(\mathbf{I}-\mathbf{P}_{v})A_{+,j}g
(\mathbf{I}-\mathbf{P}_{v})A_{-,i}hdwdv\\
&+\sum_{k=1}^{3}\sum_{i,j=1}^{3}\iint_{2| w|\leq| v|,| v|\geq1}\partial_{k}\phi^{ij}(v)w_{k}\sqrt{\mu(w)}f(w)\{\mathbf{P}_{v}A_{+,j}g
(\mathbf{I}-\mathbf{P}_{v})A_{-,i}h\\
&
\qquad\qquad+(\mathbf{I}-\mathbf{P}_{v})A_{\pm,j}g
\mathbf{P}_{v}A_{-,i}h+(\mathbf{I}-\mathbf{P}_{v})A_{+,j}g
(\mathbf{I}-\mathbf{P}_{v})A_{-,i}h\}dwdv\\
&+\frac{1}{2}\sum_{k,l=1}^{3}\sum_{i,j=1}^{3}\int_{0}^{1}\iint_{2| w|\leq| v|,| v|\geq1}\langle v\rangle^{2\theta}\partial_{kl}\phi^{ij}(v-sw)w_{k}w_{l}\\
&\qquad\qquad\times\sqrt{\mu(w)}f(w)A_{+,j}g(v)
A_{-,i}h(v)dwdvds.
\end{align*}
Since $2| w|\leq| v|,| v|\geq1,0<s<1,$ for $\gamma\ge0$, we have
\begin{align}\label{v-s}
|\partial_{kl}\phi^{ij}(v-sw)|\leq C| v-sw|^{\gamma}\leq C\langle v\rangle^{\gamma}.
\end{align}
It follows from the inequality \eqref{v-s}, the norm inequality \eqref{definition3} that
\begin{align*}
\left|\langle\Gamma_1(f, g), h\rangle\right|
&\lesssim\|f\|_{L^2(\mathbb{R}^3)}\|(\mathbf{I}-\mathbf{P}_{v})\nabla_{\mathcal{H}_{+}}g\|_{2,\frac{\gamma+2}{2}}
\|\mathbf{P}_{v}\nabla_{\mathcal{H}_{-}}h\|_{2,\frac{\gamma}{2}}\\
&\quad+\|f\|_{L^2(\mathbb{R}^3)}\|(\mathbf{I}-\mathbf{P}_{v})\nabla_{\mathcal{H}_{+}}g\|_{2,\frac{\gamma+2}{2}}\|(\mathbf{I}-\mathbf{P}_{v})\nabla_{\mathcal{H}_{-}}h\|_{2,\frac{\gamma+2}{2}}\\
&\quad+\|f\|_{L^2(\mathbb{R}^3)}\|\nabla_{\mathcal{H}_{+}}g\|_{2,\frac{\gamma}{2}}\|(\mathbf{I}-\mathbf{P}_{v})\nabla_{\mathcal{H}_{-}}h\|_{2,\frac{\gamma+2}{2}}\\
&\quad+\|f\|_{L^2(\mathbb{R}^3)}\|\nabla_{\mathcal{H}_{+}}g\|_{2,\frac{\gamma}{2}}\|\nabla_{\mathcal{H}_{-}}h\|_{2,\frac{\gamma}{2}}
\\
&\lesssim\|f\|_{L^2(\mathbb{R}^3)}\||g|\|_{\sigma}\||h|\|_{\sigma}.
\end{align*}
For the second term $\langle\Gamma_2(f, g), h\rangle$, we use an integration by parts and an commutator operation inside the convolution to get
\begin{align*}
\phi^{ij}\ast( \sqrt{\mu} A_{+,j}f)=-\phi^{ij}\ast\partial_j( \sqrt{\mu}f)=-\partial_j\phi^{ij}\ast( \sqrt{\mu}f),
\end{align*}
which implies that
\begin{align*}
\langle\Gamma_2(f, g), h\rangle
=\sum^3_{i,j=1}\langle  \{[\partial_j\phi^{ij}\ast( \sqrt{\mu}f)]g\},A_{-,i}h\rangle.
\end{align*}
Since $|\partial_j\phi^{ij}(v)|\lesssim |v|^{\gamma+1}$, by using \eqref{cov}, we have
$$|\partial_j\phi^{ij}\ast( \sqrt{\mu}f)|\lesssim \langle v\rangle^{\gamma+1}\|f\|_{L^2(\mathbb{R}^3)}.$$
Then it follows from \eqref{definition2} and \eqref{definition3} that
\begin{align*}
|\langle\Gamma_2(f, g), h\rangle|
&\lesssim \|f\|_{L^2(\mathbb{R}^3)}\sum^3_{i=1}\int_{\mathbb{R}^3}\langle v\rangle^{\gamma+1}|g||A_{-,i}h|dv\\
&\leq \|f\|_{L^2(\mathbb{R}^3)}\|g\|_{2,\frac{\gamma+2}{2}}\sum^3_{i=1}\|A_{-,i}h\|_{2,\frac{\gamma}{2}}\\
&\lesssim\|f\|_{L^2(\mathbb{R}^3)}\||g|\|_{\sigma}\||h|\|_{\sigma}.
\end{align*}
We conclude that
$$|\langle\Gamma(f, g), h\rangle|\lesssim \|f\|_{L^2(\mathbb{R}^3)}\||g|\|_{\sigma}\||h|\|_{\sigma}.$$
\end{proof}

\section{New Leibniz formula and Trilinear estimate}\label{S3}
In this section, we present a new Leibniz's formula which is crucial to prove the commutators estimate of linear and nonlinear Landau operators.
\begin{lemma}\label{leibniz}
For any $m\in \mathbb{N}$, we have
\begin{equation}\label{1}
\begin{split}
  \nabla^{m}_{\mathcal{H}_{+}}{\bf \Gamma}(f, g)
=&\sum^m_{k=0}C^k_m\sum^3_{i,j=1}A_{+,i} \{(\phi^{ij}\ast( \sqrt{\mu}\nabla^{k}_{\mathcal{H}_{+}}f) )A_{+,j}\nabla^{m-k}_{\mathcal{H}_{+}}g\} \\
&-\sum^m_{k=0}C^k_m\sum^3_{i,j=1}A_{+,i} \{(\phi^{ij}\ast( \sqrt{\mu} A_{+,j}\nabla^{k}_{\mathcal{H}_{+}}f) )\nabla^{m-k}_{\mathcal{H}_{+}}g\}.
\end{split}
\end{equation}
\end{lemma}
\begin{proof}
By using the representation $\Gamma(f,g)$ of \eqref{Gamma} in Lemma \ref{detail}, and  the fact
\[
A_{+,i}\nabla^{m}_{\mathcal{H}_{+}}=\nabla^{m}_{\mathcal{H}_{+}}A_{+,i},
\]
we have
\[
\nabla^{m}_{\mathcal{H}_{+}}{\bf \Gamma}(f, g)
=\sum^3_{i,j=1}A_{+,i} \nabla^{m}_{\mathcal{H}_{+}}\{(\phi^{ij}\ast(\sqrt{\mu}f))A_{+,j}g-(\phi^{ij}\ast(\sqrt{\mu}A_{+,j} f))g\}.
\]
Now we intend to prove
\begin{align*}
&\nabla^{m}_{\mathcal{H}_{+}}\Big([\phi^{ij}\ast(\sqrt{\mu}f)]A_{+,j}g\Big)=\sum^m_{k=0}C^k_m[\phi^{ij}\ast( \sqrt{\mu} \nabla^{k}_{\mathcal{H}_{+}}f) ]\nabla^{m-k}_{\mathcal{H}_{+}}A_{+,j}g,\\
&\nabla^{m}_{\mathcal{H}_{+}}\Big([\phi^{ij}\ast(\sqrt{\mu}A_{+,j}f)]g\Big)=\sum^m_{k=0}C^k_m[\phi^{ij}\ast( \sqrt{\mu} \nabla^{k}_{\mathcal{H}_{+}}A_{+,j}f) ]\nabla^{m-k}_{\mathcal{H}_{+}}g.
\end{align*}
So that, we only need to prove the following formula
\begin{equation}\label{leibniz-formula}
\nabla^{m}_{\mathcal{H}_{+}}\big((\phi^{ij}*(\sqrt{\mu}F))G\big)=
\sum^m_{k=0}C^k_m(\phi^{ij}*(\sqrt{\mu}\nabla^{k}_{\mathcal{H}_{+}}F))\nabla^{m-k}_{\mathcal{H}_{+}}G.
\end{equation}
We prove this formula by induction. $m=0$, the formula \eqref{leibniz-formula} is trivially true.\\
For  $m=1$. Using \eqref{creat-1}, directly calculation shows that
\begin{equation}\label{m=1}
\begin{split}
&\nabla_{\mathcal{H}_{+}}\big((\phi^{ij}*(\sqrt{\mu}F))G\big)\\
=&\Big(-(\phi^{ij}*\partial_1(\sqrt{\mu}F))G,-(\phi^{ij}*\partial_2(\sqrt{\mu}F))G,-(\phi^{ij}*\partial_3(\sqrt{\mu}F))G\Big)\\
&+(\phi^{ij}*(\sqrt{\mu}F))\nabla_{\mathcal{H}_{+}}G\\
=&\Big((\phi^{ij}*(\sqrt{\mu}A_{+,1}F))G,(\phi^{ij}*(\sqrt{\mu}A_{+,2}F))G,(\phi^{ij}*(\sqrt{\mu}A_{+,3}F))G\Big)\\
&+(\phi^{ij}*(\sqrt{\mu}F))\nabla_{\mathcal{H}_{+}}G\\
=&(\phi^{ij}*(\sqrt{\mu}\nabla_{\mathcal{H}_{+}}F))G+(\phi^{ij}*(\sqrt{\mu}F))\nabla_{\mathcal{H}_{+}}G.
\end{split}
\end{equation}
Now assume that, the equality \eqref{leibniz-formula} holds true for  $m\geq1$, we intend to prove that it is right for $m+1$.

It follows from the induction assumption and \eqref{m=1} that
\begin{align*}
&\nabla^{m+1}_{\mathcal{H}_{+}}\big((\phi^{ij}*(\sqrt{\mu}F))G\big)\\
&=\nabla_{\mathcal{H}_{+}}\left(\sum^m_{k=0}C^k_m(\phi^{ij}*(\sqrt{\mu}\nabla^{k}_{\mathcal{H}_{+}}F))\nabla^{m-k}_{\mathcal{H}_{+}}G\right)\\
&=\sum^m_{k=0}C^k_m(\phi^{ij}*(\sqrt{\mu}\nabla^{k+1}_{\mathcal{H}_{+}}F))\nabla^{m-k}_{\mathcal{H}_{+}}G\\
&\qquad+\sum^m_{k=0}C^k_m(\phi^{ij}*(\sqrt{\mu}\nabla^{k}_{\mathcal{H}_{+}}F))\nabla^{m-k=1}_{\mathcal{H}_{+}}G\\
&=\sum^{m+1}_{k=0}C^k_{m+1}(\phi^{ij}*(\sqrt{\mu}\nabla^{k}_{\mathcal{H}_{+}}F))\nabla^{m+1-k}_{\mathcal{H}_{+}}G,
\end{align*}
where we use the fact
 $$C^{k-1}_{m}+C^{k}_{m}=C^{k}_{m+1}.$$
We end the proof of Lemma \ref{leibniz}.
\end{proof}
\begin{remark}
Using \eqref{creat-2}, we have the fact
$$\phi^{ij}*(\sqrt{\mu}\nabla^{k}_{\mathcal{H}_{+}}\sqrt{\mu})=(-1)^{k}\phi^{ij}*\nabla^k\mu=(-1)^{k}\nabla^k\sigma^{ij},$$
which implies,
\begin{equation}\label{leibniz-formula2}
\nabla^{m}_{\mathcal{H}_{+}}\big(\sigma^{ij}G\big)=\sum^m_{k=0}C^k_m(-1)^{k}(\nabla^k\sigma^{ij}) \nabla^{m-k}_{\mathcal{H}_{+}}G.
\end{equation}
\end{remark}

In the following, we prepare to prove the trilinear estimates of the nonlinear Landau operators.
\begin{proposition}\label{trilinear1}
Let $f,g,h\in\mathcal{S}(\mathbb{R}^3)$, for $m\in \mathbb{N}$, $\gamma\ge0$, there is a positive constant $C_0$ which is independent on $m$, such that
\begin{align*}
|\langle \nabla^m_{\mathcal{H}_{+}}{\bf \Gamma}(f, g),\nabla^m_{\mathcal{H}_{+}}h\rangle|
&\leq C_0\|f\|_{L^2(\mathbb{R}^3)}\||\nabla^{m}_{\mathcal{H}_{+}}g|\|_{\sigma}\||\nabla^{m}_{\mathcal{H}_{+}}h|\|_{\sigma}\\
&\quad+C_0\sum^m_{k=1}C^k_m\|\nabla^{k-1}_{\mathcal{H}_{+}}f\|_{L^2(\mathbb{R}^3)}\||\nabla^{m-k}_{\mathcal{H}_{+}}g|\|_{\sigma}\||\nabla^{m}_{\mathcal{H}_{+}}h|\|_{\sigma}.
\end{align*}
\end{proposition}
\begin{proof}
One can verify that
\begin{equation*}%\label{1+}
\begin{split}
 &\langle\nabla^m_{\mathcal{H}_{+}}{\bf \Gamma}(f, g),\nabla^m_{\mathcal{H}_{+}}h\rangle-\langle{\bf \Gamma}(f,\nabla^m_{\mathcal{H}_{+}} g),\nabla^m_{\mathcal{H}_{+}}h\rangle\\
=&\sum^m_{k=1}C^k_m\sum^3_{i,j=1}\langle  \{[\phi^{ij}\ast( \sqrt{\mu} \nabla^k_{\mathcal{H}_{+}}f) ]A_{+,j}\nabla^{m-k}_{\mathcal{H}_{+}}g\},A_{-,i} \nabla^m_{\mathcal{H}_{+}}h\rangle\\
&-\sum^m_{k=1}C^k_m\sum^3_{i,j=1}\langle  \{[\phi^{ij}\ast( \sqrt{\mu} A_{+,j}\nabla^k_{\mathcal{H}_{+}}f) ]\nabla^{m-k}_{\mathcal{H}_{+}}g\},A_{-,i} \nabla^m_{\mathcal{H}_{+}}h\rangle\\
=&\mathbf{\Gamma}_{m,1}+\mathbf{\Gamma}_{m,2}.
\end{split}
\end{equation*}
Firstly, we can deduce from Proposition \ref{3-estimate} that
$$\left|\langle{\bf \Gamma}(f,\nabla^m_{\mathcal{H}_{+}} g),\nabla^m_{\mathcal{H}_{+}}h\rangle\right|\lesssim \|f\|_{L^2(\mathbb{R}^3)}\||\nabla^m_{\mathcal{H}_{+}}g|\|_{\sigma}\||\nabla^m_{\mathcal{H}_{+}}h|\|_{\sigma}.$$
For the term $\mathbf{\Gamma}_{m,1}$,  the derivative on the convolution to get,
$$\phi^{ij}\ast( \sqrt{\mu} \nabla^k_{\mathcal{H}_{+}}f)=\nabla\phi^{ij}\ast( \sqrt{\mu} \nabla^{k-1}_{\mathcal{H}_{+}}f).$$
We only need to consider the third part $\{2| w|\leq| v|,| v|\geq1\}$. Taylor Expanding $\nabla\phi^{ij}\phi^{ij}(v-w)$ to get
\begin{equation*}
\nabla\phi^{ij}(v-w)=\nabla\phi^{ij}(v)+\sum^3_{l=1}\left(\int_{0}^{1}\partial_{l}\nabla\phi^{ij}(v-sw)ds\right)w_l.
\end{equation*}
This along with the fact that
$$
\sum_{i,j=1}^{3}\partial_{p}\phi^{ij}(v)v_iv_j=-2\sum_{j=1}^{3}\phi^{pj}(v)v_j=0,\quad\forall p=1,2,3
$$
show immediately
\begin{align*}
&\sum_{i,j=1}^{3}\iint\nabla\phi^{ij}(v-w)\sqrt{\mu(w)}\nabla^{k-1}_{\mathcal{H}_{+}}f(w)A_{+,j}\nabla^{m-k}_{\mathcal{H}_{+}}g(v)
A_{-,i}\nabla^m_{\mathcal{H}_{+}}h(v)dwdv\\
=&\sum_{i,j=1}^{3}\iint\nabla\phi^{ij}(v)\sqrt{\mu(w)}\nabla^{k-1}_{\mathcal{H}_{+}}f(w)(\mathbf{I}-\mathbf{P}_{v})A_{+,j}\nabla^{m-k}_{\mathcal{H}_{+}}g(v)\mathbf{P}_{v}A_{-,i}\nabla^m_{\mathcal{H}_{+}}hdwdv\\
&+\sum_{i,j=1}^{3}\iint\nabla\phi^{ij}(v)\sqrt{\mu(w)}\nabla^{k-1}_{\mathcal{H}_{+}}f(w)A_{+,j}\nabla^{m-k}_{\mathcal{H}_{+}}g(v)(\mathbf{I}-\mathbf{P}_{v})A_{-,i}\nabla^m_{\mathcal{H}_{+}}hdwdv\\
&+\sum_{l=1}^{3}\sum_{i,j=1}^{3}\int_{0}^{1}\iint_{2| w|\leq| v|,| v|\geq1}\partial_{l}\nabla\phi^{ij}(v-sw)w_{l}\sqrt{\mu(w)}\\
&\qquad\qquad\qquad\qquad  \times  \nabla^{k-1}_{\mathcal{H}_{+}}f(w)A_{+,j}\nabla^{m-k}_{\mathcal{H}_{+}}g(v)
A_{-,i}\nabla^m_{\mathcal{H}_{+}}h(v)dwdvds.
\end{align*}
For $\gamma\ge0$, consider that
$$
|\nabla\phi^{ij}(v)|\lesssim |v|^{\gamma+1}\leq \langle v\rangle^{\gamma+1},
$$
and for $2| w|\leq| v|,| v|\geq1,0<s<1,$  we have
\begin{align*}
|\partial_{l} \nabla\phi^{ij}(v-sw)|\leq C| v-sw|^{\gamma}\lesssim\langle v\rangle^{\gamma}, \quad\forall p=1,2,3.
\end{align*}
It follows from Cauchy-Schwartz's inequality and Proposition \ref{norm} that
\begin{align*}
|\mathbf{\Gamma}_{m,1}|&\lesssim\sum^m_{k=1}C^k_m
\Big\{
\|(\mathbf{I}-\mathbf{P}_{v})\nabla^{m-k+1}_{\mathcal{H}_{+}}g\|_{2,\frac{\gamma+2}{2}}\|\mathbf{P}_{v}\nabla_{\mathcal{H}_{-}}\nabla^{m}_{\mathcal{H}_{+}}h\|_{2,\frac{\gamma}{2}}\\
&\qquad\qquad\qquad+\|\nabla^{m-k+1}_{\mathcal{H}_{+}}g\|_{2,\frac{\gamma}{2}}\|(\mathbf{I}-\mathbf{P}_{v})\nabla_{\mathcal{H}_{-}}\nabla^{m}_{\mathcal{H}_{+}}h\|_{2,\frac{\gamma+2}{2}}\\
&\qquad\qquad\qquad+\|\nabla^{m-k+1}_{\mathcal{H}_{+}}g\|_{2,\frac{\gamma}{2}}\|\nabla_{\mathcal{H}_{-}}\nabla^{m}_{\mathcal{H}_{+}}h\|_{2,\frac{\gamma}{2}}\Big\}\|\nabla^{k-1}_{\mathcal{H}_{+}}f\|_{L^2(\mathbb{R}^3)}\\
&\lesssim\sum^m_{k=1}C^k_m
\|\nabla^{k-1}_{\mathcal{H}_{+}}f\|_{L^2(\mathbb{R}^3)}\||\nabla^{m-k}_{\mathcal{H}_{+}}g|\|_{\sigma}\||\nabla^{m}_{\mathcal{H}_{+}}h|\|_{\sigma}.
\end{align*}
For the second term $\mathbf{\Gamma}_{m,2}$, we use an integration by parts and an commutator operation inside the convolution to get
\begin{align*}
\phi^{ij}\ast( \sqrt{\mu} A_{+,j}\nabla^{k}_{\mathcal{H}_{+}}f)=\partial_{j}\nabla\phi^{ij}\ast( \sqrt{\mu}\nabla^{k-1}_{\mathcal{H}_{+}}f),
\end{align*}
it implies that
\begin{align*}
\mathbf{\Gamma}_{m,2}
=-\sum^m_{k=1}C^k_m\sum^3_{i,j=1}\langle  \{[\partial_{j}\nabla\phi^{ij}\ast( \sqrt{\mu}\nabla^{k-1}_{\mathcal{H}_{+}}f)]\nabla^{m-k}_{\mathcal{H}_{+}}g\},A_{-,i} \nabla^{m}_{\mathcal{H}_{+}}h\rangle.
\end{align*}
Consider again that, for $\gamma\ge0$, $l=1,2,3$
$$|\partial_{lj}\phi^{ij}(v)|\lesssim \langle v\rangle^{\gamma},$$
using the Cauchy-Schwarz inequality, Lemma \ref{covo} and definition \eqref{definition3},  we find that
\begin{align*}
|\mathbf{\Gamma}_{m,2}|\lesssim& \sum^m_{k=1}C^k_m\|\nabla^{k-1}_{\mathcal{H}_{+}}f\|_{L^2(\mathbb{R}^3)}
\|\nabla^{m-k+1}_{\mathcal{H}_{+}}g\|_{2,\frac{\gamma}{2}}
\|\nabla_{\mathcal{H}_{-}}\nabla^{m}_{\mathcal{H}_{+}}h\|_{2,\frac{\gamma}{2}}\\
\lesssim&\sum^m_{k=1}C^k_m\|\nabla^{k-1}_{\mathcal{H}_{+}}f\|_{L^2(\mathbb{R}^3)}\||\nabla^{m-k}_{\mathcal{H}_{+}}g|\|_{\sigma}\||\nabla^{m}_{\mathcal{H}_{+}}h|\|_{\sigma}.
\end{align*}
Substituting the estimates of  $\mathbf{\Gamma}_{m,1}$ and $\mathbf{\Gamma}_{m,2}$ into \eqref{1},
we end the proof of Proposition \ref{trilinear1}.
\end{proof}
From the equality \eqref{K-A}, we set $g=\sqrt{\mu}$ in Proposition \ref{trilinear1} to get
\begin{corollary}\label{K-estimate}
Let $f,h\in\mathcal{S}(\mathbb{R}^3)$, for $m\in \mathbb{N}$,  there is a positive constant $C_0$ which is independent on $m$,  such that,
for $\gamma\ge0$,
\begin{align*}
|\langle \nabla^m_{\mathcal{H}_{+}}\mathcal{L}_2f, \nabla^m_{\mathcal{H}_{+}}h\rangle|
&\leq
C_0^{m+1}\sqrt{m!}\|f\|_{L^2(\mathbb{R}^3)}\| \nabla^m_{\mathcal{H}_{+}}h\|_{\sigma}\\
&\quad+\sum^m_{k=1}C^k_mC_0^{m-k+1}\sqrt{(m-k)!}\|\nabla^{k-1}_{\mathcal{H}_{+}}f\|_{L^2(\mathbb{R}^3)}\||\nabla^{m}_{\mathcal{H}_{+}}h|\|_{\sigma}.
\end{align*}
\end{corollary}
\begin{proof}
Since $\sqrt{\mu}=\psi_0, $  one can verify that
\begin{align*}
A_+^{\alpha}\sqrt{\mu}=\sqrt{\alpha!}\psi_{\alpha}.
\end{align*}
Then
\begin{align*}
\||\nabla^{m-k}_{\mathcal{H}_{+}}\sqrt{\mu}|\|^2_{\sigma}&=\sum_{|\alpha|=m-k}\frac{(m-k)!}{\alpha!}\||A_+^{\alpha}\sqrt{\mu}|\|^2_{\sigma}\\
&=\sum_{|\alpha|=m-k}(m-k)!\||\psi_{\alpha}|\|^2_{\sigma}\\
&\lesssim (m-k)!(m-k)^{\frac{\gamma}{2}+4} \leq C_0^{m-k}(m-k)!.
\end{align*}
We end the proof of Corollary \ref{K-estimate} by substituting $g=\sqrt{\mu}$ into the estimate of Proposition \ref{trilinear1}.
\end{proof}

%%%%%%%%%%%%%%%%%%%%%%%%%%%%%%
\section{The coercivity of linear Landau operator}\label{S4}
In this section, we show the coercivity of linear Landau operator.
On the basis of the predecessors Lemmas, the coercivity estimate for the linear Landau operator $\mathcal{L}_1$ is as follows.
\begin{proposition}\label{multiply-estimate}
Let $g\in\mathcal{S}(\mathbb{R}^3)$,  $\mathcal{L}_1$ was defined in \eqref{Gamma}, for any $m\in\mathbb{N}$, there exist a positive constant $C_{0}>0$ which is independent on $m$, such that,
\begin{align*}
\left(\nabla^{m}_{\mathcal{H}_{+}}\mathcal{L}_1g,\nabla^{m}_{\mathcal{H}_{+}}g\right)_{L^{2}(\mathbb{R}^{3})}
\geq&\| |\nabla^{m}_{\mathcal{H}_{+}}g|\|_{\sigma}^{2}-C_0\||\nabla^{m}_{\mathcal{H}_{+}}g|\|_{\sigma}\|\nabla^{m}_{\mathcal{H}_{+}}g\|_{2,\frac{\gamma}{2}}\\
&-C_0\sum^m_{k=1}k C^{k}_{m}\sqrt{k!}\| |\nabla^{m-k}_{\mathcal{H}_{+}}g|\|_{\sigma}\|| \nabla^{m}_{\mathcal{H}_{+}}g|\|_{\sigma}\\
&-C_0\sum^{m-1}_{k=1}k (m-k)C^{k}_{m}\sqrt{k!}\||\nabla^{m-k-1}_{\mathcal{H}_{+}}g|\|_{\sigma}\||\nabla^{m}_{\mathcal{H}_{+}}g|\|_{\sigma}.
\end{align*}
\end{proposition}
\begin{proof}
Recalled the formula $\mathcal{L}_1g$ in \eqref{Gamma} and $\sigma^{ij}$ in \eqref{sigma},  since $\nabla^{m}_{\mathcal{H}_{+}}A_{+,i}=A_{+,i}\nabla^{m}_{\mathcal{H}_{+}}$, integrated by parts, we have
\begin{align*}
\left(\nabla^{m}_{\mathcal{H}_{+}}\mathcal{L}_1g,\nabla^{m}_{\mathcal{H}_{+}}g\right)_{L^{2}(\mathbb{R}^{3})}
=\sum_{i,j=1}^{3}\int_{\mathbb{R}^3}\nabla^{m}_{\mathcal{H}_{+}}\Big(\sigma^{ij}A_{-,j}g\Big)A_{-,i}\nabla^{m}_{\mathcal{H}_{+}}g
dv.
\end{align*}
Then by using the new Leibniz formula \eqref{leibniz-formula2}
$$
\nabla^{m}_{\mathcal{H}_{+}}\Big(\sigma^{ij}A_{-,j}g\Big)=\sum^m_{k=0}C^{k}_{m}(-1)^{k}\nabla^{k}\sigma^{ij}\nabla^{m-k}_{\mathcal{H}_{+}}A_{-,j}g,
$$
along with the facts on the operator commutation
\begin{align*}
&[A_{+,j}, A_{-,j}]=A_{+,j}A_{-,j}-A_{-,j}A_{+,j}=-1,\\
&[A_{+,l}, A_{-,j}]=0\quad\text{if}\quad l\neq j,
\end{align*}
which means that, for $m-k\ge1$
$$\nabla^{m-k}_{\mathcal{H}_{+}}A_{-,j}=A_{-,j}\nabla^{m-k}_{\mathcal{H}_{+}}-(m-k)\nabla^{m-k-1}_{\mathcal{H}_{+}},$$
We can deduce that
\begin{equation}\label{linear-decomp}
\begin{split}
 &\left(\nabla^{m}_{\mathcal{H}_{+}}\mathcal{L}_1g,\nabla^{m}_{\mathcal{H}_{+}}g\right)_{L^{2}(\mathbb{R}^{3})}\\
=&\|| \nabla^{m}_{\mathcal{H}_{+}}g|\|_{\sigma}^{2}+\sum_{i=1}^{3}\int_{\mathbb{R}^3}\sigma^{i}\nabla^{m}_{\mathcal{H}_{+}}g\partial_i\nabla^{m}_{\mathcal{H}_{+}}gdv\\
&+\sum^m_{k=1}C^{k}_{m}(-1)^{k}\sum_{i,j=1}^{3}
\int_{\mathbb{R}^3}\nabla^{k}\sigma^{ij}A_{-,j}\nabla^{m-k}_{\mathcal{H}_{+}}gA_{-,i}\nabla^{m}_{\mathcal{H}_{+}}gdv\\
&-\sum^{m-1}_{k=1}C^{k}_{m}(-1)^{k}(m-k)\sum_{i,j=1}^{3}
\int_{\mathbb{R}^3}\nabla^{k}\sigma^{ij}
\nabla^{m-k-1}_{\mathcal{H}_{+}}gA_{-,i}\nabla^{m}_{\mathcal{H}_{+}}gdv\\
=&\| |\nabla^{m}_{\mathcal{H}_{+}}g|\|_{\sigma}^{2}+\mathbf{R}_{0}(g)+\mathbf{R}_{1}(g)+\mathbf{R}_{2}(g).
\end{split}
\end{equation}
So that the proof of Proposition \ref{multiply-estimate} is reduced to the estimations of $\mathbf{R}_0(g), \mathbf{R}_1(g)$ and $\mathbf{R}_2(g)$, which will be showed in the following three Lemmas.
\end{proof}
\begin{lemma}
For $g\in\mathcal{S}(\mathbb{R}^3)$,  for any $\alpha\in\mathbb{N}^3$,
\begin{align*}
|\mathbf{R}_{0}(g)|
\lesssim\|\nabla^{m}_{\mathcal{H}_{+}}g\|_{2,\frac{\gamma}{2}}\||\nabla^{m}_{\mathcal{H}_{+}}g|\|_{\sigma}.
\end{align*}
\end{lemma}
\begin{proof}
For the term $\mathbf{R}_{0}(g)$, we integrate by parts to get
\begin{align*}
\mathbf{R}_{0}(g)=-\mathbf{R}_{0}(g)-\sum^3_{i=1}\int_{\mathbb{R}^3}\partial_i\sigma^{i}|\nabla^{m}_{\mathcal{H}_{+}}g|^2dv.
\end{align*}
By using the fact
$$\sum^3_{i=1}\partial_i\phi^{ij}(v)=-2|v|^{\gamma}v_j,$$
it follows that
$$\sum^3_{i=1}\partial_i\sigma^{i}=-2\sum^3_{j=1}|v|^{\gamma}v_j*(v_j\mu).$$
One can deduce from Lemma \ref{covo} that
\begin{align*}
|\partial_i\sigma^{i}|\lesssim \langle v\rangle^{\gamma+1}.
\end{align*}
Then from \eqref{definition2}, we have,
\begin{align*}
|\mathbf{R}_{0}(g)|
&
\lesssim\int_{\mathbb{R}^3}\langle v\rangle^{\gamma+1}|\nabla^{m}_{\mathcal{H}_{+}}g|^2dv\\
&\leq\|\nabla^{m}_{\mathcal{H}_{+}}g\|_{2,\frac{\gamma}{2}}
\|\nabla^{m}_{\mathcal{H}_{+}}g\|_{2,\frac{\gamma+2}{2}} \lesssim\|\nabla^{m}_{\mathcal{H}_{+}}g\|_{2,\frac{\gamma}{2}}\||\nabla^{m}_{\mathcal{H}_{+}}g|\|_{\sigma}.
\end{align*}
\end{proof}

Now we estimate $R_1(g)$.
\begin{lemma}\label{R-1}
We have
\begin{align*}
|\mathbf{R}_{1}(g)|\lesssim\sum^m_{k=1}k C^{k}_{m}\sqrt{k!}\| |\nabla^{m-k}_{\mathcal{H}_{+}}g|\|_{\sigma}\|| \nabla^{m}_{\mathcal{H}_{+}}g|\|_{\sigma}.
\end{align*}
\end{lemma}
\begin{proof}
For the term $\mathbf{R}_{1}(g)$,  by using the new Leibniz formula \eqref{leibniz-formula2}
\begin{align*}
\mathbf{R}_{1}(g)&=\sum^m_{k=1}C^{k}_{m}\sum_{i,j=1}^{3}
\int_{\mathbb{R}^3}\phi\ast(\sqrt{\mu}\nabla^{k}_{\mathcal{H}_{+}}\sqrt{\mu})A_{-,j}\nabla^{m-k}_{\mathcal{H}_{+}}gA_{-,i}\nabla^{m}_{\mathcal{H}_{+}}gdv\\
&=\sum^m_{k=1}C^{k}_{m}\langle {\bf \Gamma}(\nabla^{k}_{\mathcal{H}_{+}}\sqrt{\mu}, \nabla^{m-k}_{\mathcal{H}_{+}}g),\nabla^{m}_{\mathcal{H}_{+}}g\rangle.
\end{align*}
Setting $f=\nabla^{k}_{\mathcal{H}_{+}}\sqrt{\mu}$ in Proposition \ref{3-estimate} to get
\begin{align*}
|\mathbf{R}_{1}(g)|
\lesssim \sum^m_{k=1}C^{k}_{m}\|\nabla^{k}_{\mathcal{H}_{+}}\sqrt{\mu}\|_{L^2(\mathbb{R}^3)}\||\nabla^{m-k}_{\mathcal{H}_{+}}g|\|_{\sigma}\||\nabla^{m}_{\mathcal{H}_{+}}g|\|_{\sigma}.
\end{align*}
Since $\sqrt{\mu}=\psi_0 $  and
$$A_+^{\beta}\sqrt{\mu}=\sqrt{\beta!}\psi_{\beta}$$
where $\{\psi_{\alpha}\}_{\alpha\in \mathbb{N}^3}$ are the orthonormal basis in $L^2(\mathbb{R}^3)$, then from the definition \eqref{multi-indices}, we have
\begin{align*}
\|\nabla^{k}_{\mathcal{H}_{+}}\sqrt{\mu}\|^2_{L^2(\mathbb{R}^3)}&=\sum_{|\beta|=k}\frac{k!}{\beta!}\|A_+^{\beta}\sqrt{\mu}\|^2_{L^2(\mathbb{R}^3)}\\
&=\sum_{|\beta|=k}\frac{k!}{\beta!}\beta!\|\psi_{\beta}\|^2_{L^2(\mathbb{R}^3)}=\sum_{|\beta|=k}k!=\frac{(k+1)(k+2)k!}{2},
\end{align*}
where the number of mult-indices of $|\beta|=k$ is $\frac{(k+1)(k+2)}{2}$.
It follows that
$$|\mathbf{R}_{1}(g)|\lesssim\sum^m_{k=1}k C^{k}_{m}\sqrt{k!}\| |\nabla^{m-k}_{\mathcal{H}_{+}}g|\|_{\sigma}\|| \nabla^{m}_{\mathcal{H}_{+}}g|\|_{\sigma}.$$
\end{proof}

For the term $\mathbf{R}_{2}(g)$, we need the following estimate.
\begin{lemma}\label{R-2}
For $g\in\mathcal{S}(\mathbb{R}^3)$,  we have
\begin{equation*}
\begin{split}
 |\mathbf{R}_{2}(g)|\lesssim\sum^{m-1}_{k=1} C^{k}_{m}k(m-k)\sqrt{k!}\||\nabla^{m-k-1}_{\mathcal{H}_{+}}g|\|_{\sigma}\||\nabla^{m}_{\mathcal{H}_{+}}g|\|_{\sigma}.
\end{split}
\end{equation*}
\end{lemma}
\begin{proof}
In fact, the integration
\begin{align*}
&(-1)^k\int_{\mathbb{R}^3}\nabla^{k}\sigma^{ij}
\nabla^{m-k-1}_{\mathcal{H}_{+}}gA_{-,i}\nabla^{m}_{\mathcal{H}_{+}}gdv\\
&=\iint_{\mathbb{R}^3\times \mathbb{R}^3}\phi^{ij}(v-w)\sqrt{\mu}(w)\nabla^{k}_{\mathcal{H}_{+}}\sqrt{\mu}(w)\nabla^{m-k-1}_{\mathcal{H}_{+}}g(v)
A_{-,i}\nabla^{m}_{\mathcal{H}_{+}}(v)dwdv.
\end{align*}
We decompose the integration region $[v,w]\in \mathbb{R}^3\times \mathbb{R}^3$ into three parts:
$$
\{| v|\leq1\},\ \  \{2| w|\geq| v|,| v|\geq1\},\ \ and\ \  \{2| w|\leq| v|,| v|\geq1\}.
$$
For the first part $\{| v|\leq1\}$ and the second part $\{2| w|\geq| v|,| v|\geq1\}$,
similar as the estimate in Lemma \ref{3-estimate}, one can verify that
\begin{align*}
&|
\iint_{\mathbb{R}^3\times \mathbb{R}^3}\phi^{ij}(v-w)\sqrt{\mu}(w)\nabla^{k}_{\mathcal{H}_{+}}\sqrt{\mu}(w)\nabla^{m-k-1}_{\mathcal{H}_{+}}g(v)
A_{-,i}\nabla^{m}_{\mathcal{H}_{+}}(v)dwdv|\\
&\lesssim k\sqrt{k!}\int_{\mathbb{R}^3}
\langle v\rangle^{\gamma+1}|\nabla^{m-k-1}_{\mathcal{H}_{+}}g
A_{-,i}\nabla^{m}_{\mathcal{H}_{+}}g| dv\\
&\leq k\sqrt{k!}\|\langle v\rangle^{\frac{\gamma+2}{2}}\nabla^{m-k-1}_{\mathcal{H}_{+}}g\|_{L^2}\|\langle v\rangle^{\frac{\gamma}{2}}A_{-,i}\nabla^{m-k-1}_{\mathcal{H}_{+}}g\|_{L^2}\\
&\lesssim k\sqrt{k!}\||\nabla^{m-k-1}_{\mathcal{H}_{+}}g|\|_{\sigma}\||\nabla^{m}_{\mathcal{H}_{+}}g|\|_{\sigma}.
\end{align*}
Now we consider the third part $\{2| w|\leq| v|,| v|\geq1\}.$ Expanding $\phi^{ij}(v-w)$ to get
$$
\phi^{ij}(v-w)=\phi^{ij}(v)+\sum_{l=1}^{3}\left(\int_{0}^{1}\partial_{l}\phi^{ij}(v-sw)ds\right)w_{l}.
$$
Since
$$
\sum_{i=1}^{3}\phi^{ij}v_{i}=0,
$$
we have
\begin{align*}
&(-1)^k\sum_{i=1}^{3}
\iint_{\mathcal{I}}\nabla^{k}\sigma^{ij}
\nabla^{m-k-1}_{\mathcal{H}_{+}}gA_{-,i}\nabla^{m}_{\mathcal{H}_{+}}gdv\\
=&\sum_{i=1}^{3}\iint_{\mathcal{I}}\phi^{ij}(v)\sqrt{\mu}(w)\nabla^{k}_{\mathcal{H}_{+}}\sqrt{\mu}(w)\nabla^{m-k-1}_{\mathcal{H}_{+}}g(v)
\Big((\mathbf{I}-\mathbf{P}_{v})A_{-,i}\nabla^{m}_{\mathcal{H}_{+}}g(v)\Big)dwdv\\
+&\sum_{k=1}^{3}\sum_{i=1}^{3}\int_{0}^{1}\iint_{\mathcal{I}}w_{l}\partial^{\beta}\mu(w)\partial_{l}\phi^{ij}(v-sw)\nabla^{m-k-1}_{\mathcal{H}_{+}}g(v)
A_{-,i}\nabla^{m}_{\mathcal{H}_{+}}g(v)dwdvds
\end{align*}
where $\mathcal{I}=\{2| w|\leq| v|,| v|\geq1\}$.  It follows from the inequality \eqref{v-s}, the norm inequalities \eqref{definition2} and  \eqref{definition3} that
\begin{align*}
|\mathbf{R}_{2}(g)|\lesssim \sum^{m-1}_{k=1} C^{k}_{m}k(m-k)\sqrt{k!}\||\nabla^{m-k-1}_{\mathcal{H}_{+}}g|\|_{\sigma}\||\nabla^{m}_{\mathcal{H}_{+}}g|\|_{\sigma}.
\end{align*}
We end the proof of Lemma \ref{R-2}.
\end{proof}
Substituting the estimates of $\mathbf{R}_{0}, \mathbf{R}_{1}$ and $\mathbf{R}_{2}$ into the decomposition \eqref{linear-decomp}, we end the proof of Proposition \ref{multiply-estimate}.

%%%%%%%%%%%%%%%%%%%%%%%%%%%%%%%%%%%%%%%%%%%%%%%%%
\section{Gelfand-Shilov smoothing effect for Landau equation}\label{S5}
Now we prepare to prove Theorem \ref{trick} by induction.

Let $g$ be the solution of Cauchy problem \eqref{eq-1}, that is,
$$
\partial_{t}g=-\mathcal{L}g+\mathbf{\Gamma}(g,g),\quad g|_{t=0}=g_0.
$$
By using the estimate in Proposition \ref{multiply-estimate},  Corollary \ref{K-estimate} and Proposition \ref{3-estimate} with $\alpha=0$, we have
\begin{align*}
\frac{d}{dt}\|g\|_{L^{2}(\mathbb{R}^{3})}^{2}
&=2(\partial_{t}g,g)_{L^{2}(\mathbb{R}^{3})}\\
&=
-2(\mathcal{L}_1g,g)_{L^{2}(\mathbb{R}^{3})}-2(\mathcal{L}_2g,g)_{L^{2}(\mathbb{R}^{3})}+2({\bf \Gamma}(g, g),g)_{L^{2}(\mathbb{R}^{3})}\\
&\leq-2\||g|\|_{\sigma}^{2}+2C_{0}\Big(\| g\|_{2,\frac{\gamma}{2}}+\|g\|_{L^2(\mathbb{R}^3)}\Big)\||g|\|_{\sigma}+2C_{0}\|g\|_{L^2(\mathbb{R}^3)}\||g|\|_{\sigma}^{2}.
\end{align*}
For $\epsilon_0$ small enough, using \eqref{1.7+01},
\begin{equation}\label{small}
\|g\|_{L^2(\mathbb{R}^3)}\leq \epsilon_0\leq \frac{1}{8C_0},
\end{equation}
then it follows that
$$
\frac{d}{dt}\|g\|_{L^{2}(\mathbb{R}^3)}^{2}+
\frac{7}{4}\||g|\|_{\sigma}^{2}\leq2C_{0}\Big(\| g\|_{2,\frac{\gamma}{2}}+\|g\|_{L^2(\mathbb{R}^3)}\Big)\||g|\|_{\sigma}.
$$
For $\gamma\ge0$, by using H\"older's inequality and the inequality \eqref{definition2}, we have
\begin{align}
\| g\|_{2,\frac{\gamma}{2}}^{2}&\leq C_{\gamma, \delta}\|g\|^2_{L^2}+\delta\|\langle v\rangle^{\frac{\gamma+2}{2}} g\|^2_{L^2(\mathbb{R}^3)}\nonumber\\
&\leq C_{\gamma,\delta}\|g\|^2_{L^2}+\delta\||g|\|^2_{\sigma},\label{interp-1}
\end{align}
Set $\delta$ small, such that
$$
\frac{d}{dt}\|g\|_{L^{2}(\mathbb{R}^3)}^{2}+
\||g|\|_{\sigma}^{2}\leq\left(4C^2_0+2C_{0}C_{\gamma}\right)\| g\|_{L^{2}(\mathbb{R}^3)}^{2}.
$$
By using the Gronwall inequality, for any $T>0$ and $0<t<T$, we have
\begin{equation*}%\label{induction-0}
\|g(t)\|_{L^{2}(\mathbb{R}^3)}^{2}+\int_{0}^{t}\| |g(\tau)|\|_{\sigma}^{2}d\tau\leq e^{\left(4C^2_0+2C_{0}C_{\gamma}\right)t}\| g_{0}\|_{L^{2}(\mathbb{R}^3)}^{2}.
\end{equation*}
So that \eqref{1.7+1} hold true for $\alpha=0$ with
\begin{equation}\label{C-0}
C^2\ge e^{\left(4C^2_0+2C_{0}C_{\gamma}\right)}\| g_{0}\|_{L^{2}(\mathbb{R}^3)}^{2}.
\end{equation}
Moreover, we can find that
\begin{align*}
&\frac{d}{dt}\|t^{\frac{1}{2}}\nabla_{\mathcal{H}_{+}} g\|_{L^{2}(\mathbb{R}^3)}^{2}\\
&=2\langle \nabla_{\mathcal{H}_{+}}\partial_{t}g,
t\nabla_{\mathcal{H}_{+}} g\rangle
+\|\nabla_{\mathcal{H}_{+}} g\|_{L^{2}(\mathbb{R}^3)}^{2}\\
&=2\langle \nabla_{\mathcal{H}_{+}}(- \mathcal{L} g+\Gamma (g, g)),
t\nabla_{\mathcal{H}_{+}} g\rangle
+\|\nabla_{\mathcal{H}_{+}} g\|_{L^{2}(\mathbb{R}^3)}^{2},
\end{align*}
and recalled the definition \eqref{multi-indices}, we can deduce from Proposition \ref{multiply-estimate}, Proposition \ref{trilinear1} and Corollary \ref{K-estimate} that
\begin{align*}
\frac{d}{dt}\|t^{\frac{1}{2}}\nabla_{\mathcal{H}_{+}} g\|_{L^{2}(\mathbb{R}^3)}^{2}
&\leq-2\||t^{\frac{1}{2}}\nabla_{\mathcal{H}_{+}} g|\|_{\sigma}^{2}+4C_{0}t\| |g|\|_{\sigma}\||\nabla_{\mathcal{H}_{+}} g|\|_{\sigma}+\|\nabla_{\mathcal{H}_{+}} g\|_{L^{2}(\mathbb{R}^3)}^{2}\\
&\quad +2C_0t\|g\|_{L^2( \mathbb{R}^3)}\||\nabla_{\mathcal{H}_{+}} g|\|^2_{\sigma}+2C_0t\|g\|_{L^2( \mathbb{R}^3)}\|| g|\|_{\sigma}\||\nabla_{\mathcal{H}_{+}} g|\|_{\sigma}\\
&\quad+2C_0t\|g\|_{L^2( \mathbb{R}^3)}\||\nabla_{\mathcal{H}_{+}} g|\|_{\sigma}+2C_0t\|g\|_{L^2( \mathbb{R}^3)}\||\nabla_{\mathcal{H}_{+}} g|\|_{\sigma}.
\end{align*}
Since for $\gamma\ge0$,
$$\|\nabla_{\mathcal{H}_{+}} g\|_{L^{2}(\mathbb{R}^3)}^{2}=\sum_{|\alpha|=1}\|A_+^{\alpha} g\|_{L^{2}(\mathbb{R}^3)}^{2}\leq \sum_{|\alpha|=1}\|\langle v\rangle^{\frac{\gamma}{2}} A_+^{\alpha} g\|_{L^{2}(\mathbb{R}^3)}^{2}\leq \frac{1}{C_1}\| |g|\|^2_{\sigma}, $$
by using the assumption \eqref{small}, we can deduce from the  Cauchy-Schwarz's inequality  that
\begin{align*}
\frac{d}{dt}\|t^{\frac{1}{2}}A_+^{\alpha} g\|_{L^{2}(\mathbb{R}^3)}^{2}+\||t^{\frac{1}{2}}A_+^{\alpha} g|\|_{\sigma}^{2}\leq 100C_0^2t\||g|\|^2_{\sigma}+\frac{1}{C_1}\| |g|\|^2_{\sigma}+t\|| g|\|^2_{\sigma}.
\end{align*}
For $0<t\le 1$, one can verify that
\begin{align*}
&\|t^{\frac{1}{2}}\nabla_{\mathcal{H}_{+}} g(t)\|_{L^{2}(\mathbb{R}^3)}^{2}
+\int_{0}^{t}\||s^{\frac{1}{2}}\nabla_{\mathcal{H}_{+}} g(s)|\|_{\sigma}^{2}ds\\
&\leq \left(100C_0^2+\frac{1}{C_1}+1\right)\int^{t}_0\| |g(s)|\|^2_{\sigma}ds\\
&\leq \left(100C_0^2+\frac{1}{C_1}+1\right)e^{\left(4C^2_0+2C_{0}C_{\gamma}\right)}\| g_{0}\|_{L^{2}(\mathbb{R}^3)}^{2}.
\end{align*}
Set
\begin{equation}\label{C-1}
C^2\ge 1000 \max\left\{e^{\left(4C^2_0+2C_{0}C_{\gamma}\right)}\| g_{0}\|_{L^{2}(\mathbb{R}^3)}^{2}, 100C_0^2+\frac{1}{C_1}+1\right\}\gg1000,
\end{equation}
we have, for $0<t\le 1$,
\begin{equation*}%\label{induction-1}
\|t^{\frac{1}{2}}\nabla_{\mathcal{H}_{+}} g(t)\|_{L^{2}(\mathbb{R}^3)}^{2}
+\int_{0}^{t}\||s^{\frac{1}{2}}\nabla_{\mathcal{H}_{+}} g(s)|\|_{\sigma}^{2}ds\leq C^2.
\end{equation*}
\begin{proposition}\label{hard}
Let $g$ be the smooth solution of the Cauchy problem \eqref{eq-1} with the assumption \eqref{small}, then there exists
$C>0$ such that for any $ n\in \mathbb{N}_+$ and $0<t\le 1$,
\begin{equation}\label{Assumption}
\|{t}^{\frac{n}{2}}\nabla^n_{\mathcal{H}_{+}} g({t})\|_{L^{2}(\mathbb{R}^3)}^{2}
+\int_{0}^{{t}}\||\tau^{\frac{n}{2}}\nabla^n_{\mathcal{H}_{+}} g(\tau)|\|_{\sigma}^{2}d\tau\leq C^{2n}((n-1)!)^{2}.
\end{equation}
\end{proposition}

\begin{proof}
In fact, we have proved that the assumption \eqref{Assumption} holds for $n=1$.
Now take take $m\ge 2$, and assume that the assumption \eqref{Assumption} holds true for $n\leq m-1$, we need to prove that the validity of \eqref{Assumption} for $n=m$. Using the equation \eqref{eq-1}, and $g\in
C^\infty(]0, +\infty[, \mathcal{S}(\mathbb{R}^3))$ is a smooth solution of the Cauchy problem \eqref{eq-1}, we have
\begin{align*}
\frac{d}{dt}\|t^{\frac{m}{2}}\nabla^m_{\mathcal{H}_{+}} g\|_{L^{2}(\mathbb{R}^3)}^{2}
=&2t^{m}\langle \nabla^m_{\mathcal{H}_{+}}\partial_{t}g,
\nabla^m_{\mathcal{H}_{+}} g\rangle+mt^{m-1}\|\nabla^m_{\mathcal{H}_{+}} g\|_{L^{2}(\mathbb{R}^3)}^{2}\\
=&-2t^{m}\langle \nabla^m_{\mathcal{H}_{+}} \mathcal{L}_1 g,
\nabla^m_{\mathcal{H}_{+}} g\rangle-2t^{m}\langle \nabla^m_{\mathcal{H}_{+}} \mathcal{L}_2 g,
\nabla^m_{\mathcal{H}_{+}} g\rangle\\
&+2t^{m}\langle \nabla^m_{\mathcal{H}_{+}} \Gamma(g, g),
\nabla^m_{\mathcal{H}_{+}} g\rangle
+mt^{m-1}\|\nabla^m_{\mathcal{H}_{+}} g\|_{L^{2}(\mathbb{R}^3)}^{2}.
\end{align*}
By using Proposition \ref{multiply-estimate} and the inequality \eqref{definition3}, we have
\begin{align*}
&\left(\nabla^m_{\mathcal{H}_{+}}\mathcal{L}_1g,\nabla^m_{\mathcal{H}_{+}}g\right)_{L^{2}(\mathbb{R}^{3})}\\
\geq&\| |\nabla^m_{\mathcal{H}_{+}}g|\|_{\sigma}^{2}-C_0\left(1+\frac{1}{C_1}\right)\sum^m_{k=1}k C^{k}_{m}\sqrt{k!}\| |\nabla^{m-k}_{\mathcal{H}_{+}}g|\|_{\sigma}\|| \nabla^{m}_{\mathcal{H}_{+}}g|\|_{\sigma}\\
&-C_0\sum^{m-1}_{k=1}k (m-k)C^{k}_{m}\sqrt{k!}\||\nabla^{m-k-1}_{\mathcal{H}_{+}}g|\|_{\sigma}\||\nabla^{m}_{\mathcal{H}_{+}}g|\|_{\sigma}.
\end{align*}
Using Corollary \ref{K-estimate} for the estimate of $\langle \nabla^{m}_{\mathcal{H}_{+}} \mathcal{L}_2 g,
\nabla^{m}_{\mathcal{H}_{+}} g\rangle$, and Proposition \ref{trilinear1} for the terms $\langle \nabla^{m}_{\mathcal{H}_{+}} \Gamma(g, g),
\nabla^{m}_{\mathcal{H}_{+}} g\rangle$, we get
\begin{align*}
&\frac{d}{dt}\|t^{\frac{m}{2}}\nabla^{m}_{\mathcal{H}_{+}} g\|_{L^{2}(\mathbb{R}^3)}^{2}+2\||t^{\frac{m}{2}}\nabla^{m}_{\mathcal{H}_{+}} g|\|_{\sigma}^{2}\\
&\quad\le mt^{m-1}\|\nabla^{m}_{\mathcal{H}_{+}} g\|_{L^{2}(\mathbb{R}^3)}^{2}\\
&\quad\quad+2C_0\left(1+\frac{1}{C_1}\right)t^{m}\sum^m_{k=1}k C^{k}_{m}\sqrt{k!}\| |\nabla^{m-k}_{\mathcal{H}_{+}}g|\|_{\sigma}\|| \nabla^{m}_{\mathcal{H}_{+}}g|\|_{\sigma}\\
&\quad\quad+2C_0t^{m}\sum^{m-1}_{k=1}k (m-k)C^{k}_{m}\sqrt{k!}\||\nabla^{m-k-1}_{\mathcal{H}_{+}}g|\|_{\sigma}\||\nabla^{m}_{\mathcal{H}_{+}}g|\|_{\sigma}\\
&\quad\quad+2t^{m}C_0^{m+1}\sqrt{m!}\|g\|_{L^2(\mathbb{R}^3)}\| \nabla^m_{\mathcal{H}_{+}}g\|_{\sigma}\\
&\qquad+2t^{m}\sum^m_{k=1}C^k_mC_0^{m-k+1}\sqrt{(m-k)!}\|\nabla^{k-1}_{\mathcal{H}_{+}}g\|_{L^2(\mathbb{R}^3)}\||\nabla^{m}_{\mathcal{H}_{+}}g|\|_{\sigma}\\
&\quad\quad+2C_0t^{m}\|g\|_{L^2(\mathbb{R}^3)}\||\nabla^{m}_{\mathcal{H}_{+}}g|\|^2_{\sigma}\\
&\quad\quad+ 2C_0t^m\sum^m_{k=1}C^k_m\|\nabla^{k-1}_{\mathcal{H}_{+}}g\|_{L^2(\mathbb{R}^3)}\||\nabla^{m-k}_{\mathcal{H}_{+}}g|\|_{\sigma}\||\nabla^{m}_{\mathcal{H}_{+}}g|\|_{\sigma},
\end{align*}
we have then, for all $0<t\le 1, m\ge 2$,
\begin{align}\label{5+110}
\|t^{\frac{m}{2}}\nabla^{m}_{\mathcal{H}_{+}} g\|_{L^{2}(\mathbb{R}^3)}^{2}
+2\int^t_0\|| \tau^{\frac{m}{2}} \nabla^{m}_{\mathcal{H}_{+}} g(\tau)|\|_{\sigma}^{2}d\tau\le \sum^7_{j=1} M_j.
\end{align}
For the term $M_1$, since $\gamma\ge0$, it follows from the inequality \eqref{definition3} that,
\begin{align*}
M_1=& m\ \int^t_0 \tau^{m-1}\|\nabla^{m}_{\mathcal{H}_{+}} g(\tau)\|_{L^{2}(\mathbb{R}^3)}^{2}d\tau\\
&\leq m\ \int^t_0 \tau^{m-1}\|\langle v\rangle^{\frac{\gamma}{2}}\nabla^{m}_{\mathcal{H}_{+}} g(\tau)\|_{L^{2}(\mathbb{R}^3)}^{2}d\tau\\
&\leq \frac{m}{C_1}\int^t_0 \tau^{m-1}\||\nabla^{m-1}_{\mathcal{H}_{+}} g(\tau)|\|^{2}_{\sigma}d\tau.
\end{align*}
Using the induction hypothesis \eqref{Assumption} for $m-1$,
we have
\begin{equation}\label{5+111}
M_1\le \frac{m C^{2m-2}}{C_1}((m-2)!)^2\leq \frac{2}{C_1}C^{2m-2}((m-1)!)^2.
\end{equation}
For the term $M_2$, we have
\begin{align*}
M_2=&2C_0\left(1+\frac{1}{C_1}\right)\sum^m_{k=1}k C^{k}_{m}\sqrt{k!}\int^t_0\tau^{m}\| |\nabla^{m-k}_{\mathcal{H}_{+}}g|\|_{\sigma}(\tau)\|| \nabla^{m}_{\mathcal{H}_{+}}g(\tau)|\|_{\sigma}d\tau\\
&\le 2C_0\left(1+\frac{1}{C_1}\right)\sum^m_{k=1}k C^{k}_{m}\sqrt{k!}\left(\int^t_0\| |\tau^{\frac{{|m-k|}}{2}}\nabla^{m-k}_{\mathcal{H}_{+}}g(\tau)|\|^2_{\sigma}d\tau\right)^{1/2}\\
&\qquad\qquad\qquad\qquad\qquad\qquad\times\left(\int^t_0\| |\tau^{\frac{m}{2}}\nabla^{m}_{\mathcal{H}_{+}}g(\tau)|\|^2_{\sigma}d\tau\right)^{1/2}
\\
&\le 2C_0\left(1+\frac{1}{C_1}\right)\sum^m_{k=1}k C^{k}_{m}\sqrt{k!}\left(C^{m-k}(m-k-1)!
\right)\\
&\qquad\qquad\qquad\qquad\qquad\qquad\times\left(\int^t_0\| |\tau^{\frac{m}{2}}\nabla^{m}_{\mathcal{H}_{+}}g(\tau)|\|^2_{\sigma}d\tau\right)^{1/2}\\
&\le \left(8C_0\left(1+\frac{1}{C_1}\right)\sum^m_{k=1}k C^{k}_{m}\sqrt{k!}C^{m-k}(m-k-1)!\right)^2\\
&\qquad\qquad\qquad\qquad\qquad\qquad+\frac 18\int^t_0\| |\tau^{\frac{m}{2}}\nabla^{m}_{\mathcal{H}_{+}}g(\tau)|\|^2_{\sigma}d\tau
\end{align*}
Since
$$
\sum^m_{k=1}\frac{C^{-k+1}m}{\sqrt{k!}(m-k)}\le 4,
$$
So that
\begin{equation}\label{5+112}
M_2\le (32C_0)^2 \left(1+\frac{1}{C_1}\right)^2C^{2m-2}((m-1)!)^2+\frac 18\int^t_0\| |\tau^{\frac{m}{2}}\nabla^{m}_{\mathcal{H}_{+}}g(\tau)|\|^2_{\sigma}d\tau.
\end{equation}
For the term $M_3$, we have
\begin{align*}
M_3=&2C_0\sum^{m-1}_{k=1}k (m-k)C^{k}_{m}\sqrt{k!}
\int^t_0\tau^{m}\|| \nabla^{m-k-1}_{\mathcal{H}_{+}}g|\|_{\sigma}\| |\nabla^{m}_{\mathcal{H}_{+}}g|\|_{\sigma}d\tau\\
&\le \Big(8C_0\sum^{m-1}_{k=1}k (m-k)C^{k}_{m}\sqrt{k!}C^{m-k-1}(m-k-2)!\Big)
\\
&\qquad+\frac 18\int^t_0\| |\tau^{\frac{m}{2}}\nabla^{m}_{\mathcal{H}_{+}}g(\tau)|\|^2_{\sigma}d\tau.
\end{align*}
Since
$$
\sum^{m-1}_{k=1}\frac{C^{-k}km}{\sqrt{k!}(m-k)}\le 8,
$$
So that
\begin{equation}\label{5+113}
M_3\le (64C_0 )^2 C^{2m-2}((m-1)!)^2+\frac 18\int^t_0\| |\tau^{\frac{m}{2}}\nabla^{m}_{\mathcal{H}_{+}}g(\tau)|\|^2_{\sigma}d\tau.
\end{equation}
For the term $M_4$, we have
\begin{align*}
M_4=&2\int^t_0\tau^{m}C_0^{m+1}\sqrt{m!}\|g\|_{L^2(\mathbb{R}^3)}
\||\nabla^{m}_{\mathcal{H}_{+}}g|\|_{\sigma}d\tau\\
&\le \left(4 C_0^{m+1}\sqrt{m!}\|g\|_{L^\infty([0, 1], L^2(\mathbb{R}^3))}\right)^2
+\frac{1}{8} \int^t_0\tau^{m}\||\nabla^{m}_{\mathcal{H}_{+}}g|\|^2_{\sigma}d\tau.
\end{align*}
By using the assumption \eqref{small}, it follows that
\begin{equation}\label{5+115}
M_4\le C_0^{2m}m!
+\frac{1}{8} \int^t_0\tau^{m}\||\nabla^{m}_{\mathcal{H}_{+}}g|\|^2_{\sigma}d\tau.
\end{equation}
For the term $M_5$, we have
\begin{align*}
M_5=&2\int^t_0\tau^{m}\sum^m_{k=1}C^k_mC_0^{m-k+1}\sqrt{(m-k)!}\|\nabla^{k-1}_{\mathcal{H}_{+}}g(\tau)\|_{L^2(\mathbb{R}^3)}\||\nabla^{m}_{\mathcal{H}_{+}}g(\tau)|\|_{\sigma}d\tau,
\end{align*}
then
\begin{equation*}
M_5\le \left(8\sum^m_{k=1}C^k_mC_0^{m-k+1}\sqrt{(m-k)!}C^{k-1}(k-2)!\right)^2
+\frac{1}{8} \int^t_0\tau^{m}\||\nabla^{m}_{\mathcal{H}_{+}}g|\|^2_{\sigma}d\tau.
\end{equation*}
since
$$
\sum^m_{k=1}\frac{m}{\sqrt{(m-k)!}k^2}\le 4,
$$
we have
\begin{equation}\label{5+117}
M_5\le (32 C_0)^2 C^{2m-2}((m-1)!)^2
+\frac{1}{8} \int^t_0\tau^{m}\||\nabla^{m}_{\mathcal{H}_{+}}g|\|^2_{\sigma}d\tau.
\end{equation}
For the term $M_6$,  using \eqref{small}, we have
\begin{equation}\label{5+114}
M_6=2C_0\int^t_0\tau^{m}\|g\|_{L^2(\mathbb{R}^3)}\||\nabla^{m}_{\mathcal{H}_{+}}g|\|^2_{\sigma}d\tau
\le \frac{1}{4} \int^t_0\tau^{m}\||\nabla^{m}_{\mathcal{H}_{+}}g|\|^2_{\sigma}d\tau
\end{equation}
For the term $M_7$, we have
\begin{align*}
&M_7=2C_0 \sum^m_{k=1}C^k_m
\int^t_0\tau^m \|\nabla^{k-1}_{\mathcal{H}_{+}}g\|_{L^2(\mathbb{R}^3)}\||\nabla^{m-k}_{\mathcal{H}_{+}}g|\|_{\sigma}\||\nabla^{m}_{\mathcal{H}_{+}}g|\|_{\sigma}d\tau\\
\le &2C_0 \sum^m_{k=1}C^k_m\|\tau^{\frac{k-1}{2}}\nabla^{k-1}_{\mathcal{H}_{+}}g\|_{L^\infty([0, 1], L^2(\mathbb{R}^3))}\\
&\qquad\times\left(\int^t_0\tau^{|m-k|}
\||\nabla^{m-k}_{\mathcal{H}_{+}}g|\|^2_{\sigma}d\tau\right)^{1/2}
\left(\int^t_0\tau^{m}\||\nabla^{m}_{\mathcal{H}_{+}}g|\|_{\sigma}d\tau\right)^{1/2}\\
\le &2C_0 \sum^m_{k=1}C^k_m\left(C^{m-1}(k-2)!(m-k-1)!\right)
\left(\int^t_0\tau^{|\alpha|}\||\nabla^{m}_{\mathcal{H}_{+}}g|\|_{\sigma}d\tau\right)^{1/2}\\
\le &
\left(16C_0\Big(1+\sum^{m-1}_{k=1}\frac{m}{(m-k)k^2}\Big)\right)^2C^{2m-2}((m-1)!)^2
+\frac{1}{8}\int^t_0\tau^{|\alpha|}\||\nabla^{m}_{\mathcal{H}_{+}}g|\|_{\sigma}d\tau.
\end{align*}
Since
$$
\sum^{m-1}_{k=1}\frac{m}{(m-k)k^2}\le 8,
$$
so that
\begin{equation}\label{5+116}
M_7\le \left(144C_0 \right)^2
C^{2m-2}((m-1)!)^2
+\frac{1}{8} \int^t_0\tau^{m}\||\nabla^{m}_{\mathcal{H}_{+}}g|\|^2_{\sigma}d\tau.
\end{equation}
then combine \eqref{5+110}-\eqref{5+116},  we get then, for $0<t\le 1, m\ge 2$,
\begin{align*}
&\|t^{\frac{m}{2}}\nabla^{m}_{\mathcal{H}_{+}} g\|_{L^{2}(\mathbb{R}^3)}^{2}
+\int^t_0\|| \tau^{\frac{m}{2}} \nabla^{m}_{\mathcal{H}_{+}} g(\tau)|\|_{\sigma}^{2}d\tau\\
&\le 8\left(1+\frac{1}{C_1}\right)^2(144 C_0)^2 C^{2m-2}((m-1)!)^2.
\end{align*}
Choose the $C$ satisfies \eqref{C-0},  \eqref{C-1} and
$$
C^2\ge 8 \left(1+\frac{1}{C_1}\right)^2(144 C_0)^2  ,
$$
we end the proof of Proposition \ref{hard}.
\end{proof}
From the result of Proposition \ref{hard},  we end the proof of Theorem \ref{trick} for $0<t\le1$.
Once we get the analytical Gelfand-Shilov of $g$ at $t=1$, then under the global smallness assumption
\eqref{1.7+01}, the proof of the propagation of analytical Gelfand-Shilov to times interval $[1, 2]$ is exactly same as the proof of Proposition \ref{hard} without the initial datum cut-off factor $t^m$, the same argument to prove the analytical Gelfand-Shilov smooth of $g$ on $[k, k+1]$ for any $k\in\mathbb{N}$.

%%%%%%%%%%%%%%%%%%%%%%%%%%%%%%
%%%%%%%%%%%%%%%%%%%%%%%%%%%%%%
\section{Appendix}\label{Appendix}
\subsection*{Hermite functions}
The standard Hermite functions $(\varphi_{n})_{n\in \mathbb{N}}$ are defined for $v\in\mathbb{R}$,
\begin{align*}%\label{H1}
\varphi_{n}(v)=\frac{(-1)^n}{\sqrt{2^nn!\sqrt{\pi}}}e^{\frac{v^2}{2}}\frac{d^n}{dv^n}(e^{-\frac{v^2}{2}})
=-\frac{1}{\sqrt{2^nn!\sqrt{\pi}}}(v-\frac{d}{dv})^n(e^{-\frac{v^2}{2}})=\frac{a_+^n\varphi_{0}}{\sqrt{n!}},
\end{align*}
where $a_+$ is the creation operator
$$
a_+=\frac{1}{\sqrt{2}}\Big(v-\frac{d}{dv}\Big).
$$
The family $(\varphi_{n})_{n\in \mathbb{N}}$ is an orthonormal basis of $L^2(\mathbb{R})$. we set for $n\geq0$, $\alpha=(\alpha_1,\alpha_2,\alpha_3)\in \mathbb{N}^3$, $x\in \mathbb{R}$,$v\in \mathbb{R}^3$,\\
  \begin{equation*}
  \phi_{n}(x)=2^{-1/4}\varphi_{n}(2^{-1/2}x),\ \ \ \
 \phi_{n}=\frac{1}{\sqrt{n!}}\left(\frac{x}{2}-\frac{d}{dx}\right)^{n}\phi_{0},
 \end{equation*}
 \begin{equation*}%\label{H2}
 \psi_{\alpha}(v)=\prod_{j=1}^{3}\phi_{\alpha_{j}}(v_j),\ \ \ \
\mathcal{E}_{k}=Span({\psi_{\alpha}})_{\alpha\in N^3,|\alpha|=k},
 \end{equation*}
with $|\alpha|=\alpha_{1}+\alpha_{2}+\alpha_3$.  The family $(\psi_{\alpha})_{\alpha\in \mathbb{N}^3}$ is an orthonormal basis of $ L^{2}(\mathbb{R}^3)$ composed by the eigenfunctions of the 3-dimensional harmonic oscillator
 \begin{equation*}%\label{decomH}
\mathcal{H}=-\Delta_{v}+\frac{|v^{2}|}{4}=\sum_{k\geq0}(k+\frac{3}{2})\mathbb{P}_{k},\quad \mathbf{Id}=\sum_{k\geq0}\mathbb{P}_{k},
 \end{equation*}
 where $\mathbb{P}_{k}$ stands for the orthogonal projection
$$
\mathbb{P}_kf=\sum_{|\alpha|=k}(f,\psi_{\alpha})_{L^2(\mathbb{R}_v)}\psi_{\alpha}.
$$
In particular,
\begin{equation*}
 \psi_{0}(v)=\frac{1}{(2\pi)^{\frac{3}{4}}}e^{-\frac{|v|^{2}}{4}}=\mu^{1/2}(v),
 \end{equation*}
where $\mu(v)$ is the Maxwellian distribution.    Setting
\begin{equation}\label{H3b}
A_{\pm,j}=\frac{v_{j}}{2}\mp\partial_j,\quad 1\leq j\leq 3,
\end{equation}
we have
\begin{equation*}
\psi_{\alpha}=\frac{1}{\sqrt{\alpha_{1}!\alpha_{2}!\alpha_{3}!}}A^{\alpha_{1}}_{+,1}A^{\alpha_{2}}_{+,2}A^{\alpha_{3}}_{+,3}\psi_{0},\quad  \alpha=(\alpha_{1},\alpha_{2},\alpha_{3})\in \mathbb{N}^{3},
\end{equation*}
and
\begin{equation}\label{H4}
A_{+,j}\psi_{\alpha}=\sqrt{\alpha_{j}+1}\psi_{\alpha+e_{j}},\quad A_{-,j}\psi_{\alpha}=\sqrt{\alpha_{j}}\psi_{_{\alpha-e_{j}}}
(=0 \, if\,  \alpha_{j}=0),
\end{equation}
where $(e_{1},e_2,e_{3})$ stands for the canonical basis of $\mathbb{R}^{3}$.  For more details of the Hermite functions, we can refer to \cite{MPX} and the reference theorem.

\subsection*{Gelfand-Shilov space}
The symmetric Gelfand-Shilov space $S^{\nu}_{\nu}(\mathbb{R}^3)$ can be characterized through the decomposition
into the Hermite basis $\{\psi_{\alpha}\}_{\alpha\in\mathbb{N}^3}$ and the harmonic oscillator $\mathcal{H}=-\triangle +\frac{|v|^2}{4}$.
For more details, see Theorem 2.1 in \cite{GPR}
\begin{align*}
f\in S^{\nu}_{\nu}(\mathbb{R}^3)&\Leftrightarrow\,f\in C^\infty (\mathbb{R}^3),\exists\, \tau>0, \|e^{\tau\mathcal{H}^{\frac{1}{2\nu}}}f\|_{L^2}<+\infty;\\
&\Leftrightarrow\, f\in\,L^2(\mathbb{R}^3),\exists\, \epsilon_0>0,\,\,\Big\|\Big(e^{\epsilon_0|\alpha|^{\frac{1}{2\nu}}}(f,\,\psi_{\alpha})_{L^2}\Big)_{\alpha\in\mathbb{N}^3}\Big\|_{l^2}<+\infty;\\
&\Leftrightarrow\,\exists\,C>0,\,A>0,\,\,\|(-\triangle +\frac{|v|^2}{4})^{\frac{k}{2}}f\|_{L^2(\mathbb{R}^3)}\leq AC^k(k!)^{\nu},\,\,\,k\in\mathbb{N}
\end{align*}
where
$$\psi_{\alpha}(v)=\phi_{\alpha_1}(v_1)\phi_{\alpha_2}(v_2)\phi_{\alpha_3}(v_3),\,\,\alpha\in\mathbb{N}^3,$$
and for $x\in\mathbb{R}$,
$$\phi_{n}(x)=\frac{(-1)^n}{\sqrt{2^nn!\pi}}e^{\frac{x^2}{2}}\frac{d^n}{dx^n}(e^{-x^2})
=\frac{1}{\sqrt{2^nn!\pi}}\Big(x-\frac{d}{dx}\Big)^n(e^{-\frac{x^2}{2}}).$$
For the harmonic oscillator $\mathcal{H}=-\triangle +\frac{|v|^2}{4}$ of 3-dimension and $s>0$, we have
\begin{equation}\label{H5}
\mathcal{H}^{\frac{k}{2}} \psi_{\alpha} = (\lambda_{\alpha})^{\frac{k}{2}}\psi_{\alpha},\,\, \lambda_{\alpha}=\sum^3_{j=1}(\alpha_j+\frac{1}{2}),\,\,k\in\mathbb{N},\,\alpha\in\mathbb{N}^3.
\end{equation}

In the following, we prove first a fundamental result in the $L^2(\mathbb{R}^3)$, which will be used to prove that the estimate \eqref{1.7+1} implies $g(t)\in S^1_1(\mathbb{R}^3)$ for any $t>0$.
\begin{lemma}\label{partial-1}
Let $u\in\mathcal{S}(\mathbb{R}^3)$, we have
\begin{equation*}
\|A_{-,j}u\|^2_{L^2(\mathbb{R}^3)}\leq\|A_{+,j}u\|^2_{L^2(\mathbb{R}^3)}.
\end{equation*}
\end{lemma}
\begin{proof}
The family $(\psi_{\alpha})_{\alpha\in \mathbb{N}^3}$ is an orthonormal basis of $L^2(\mathbb{R}^3)$, we identify $u\in\mathcal{S}(\mathbb{R}^3)$ with
$$u=\sum_{\alpha\in \mathbb{N}^3}(u,\psi_{\alpha})_{L^2(\mathbb{R}^3)}\psi_{\alpha}.$$
Denote $u_{\alpha}=(u,\psi_{\alpha})_{L^2(\mathbb{R}^3)},$   we have
\begin{align*}
&A_{-,j}u
=\sum_{\alpha\in \mathbb{N}^3}u_{\alpha}\sqrt{\alpha_j}\psi_{\alpha-e_j},\\
&A_{+,j}u=\sum_{\alpha\in \mathbb{N}^3}u_{\alpha}\sqrt{\alpha_j+1}\psi_{\alpha+e_j}.
\end{align*}
By using orthogonal of the basis $\psi_{\alpha}$, one can verify that
\begin{align*}
\|A_{-,j}u\|^2_{L^2(\mathbb{R}^3)}
&=\sum_{\alpha\in \mathbb{N}^3}|u_{\alpha}|^2\alpha_j\\
&\leq\sum_{\alpha\in \mathbb{N}^3}|u_{\alpha}|^2\Big(\alpha_j+1\Big)
=\|A_{+,j}u\|^2_{L^2(\mathbb{R}^3)}.
\end{align*}
We end the proof of Lemma \ref{partial-1}.
\end{proof}
On the basis of Lemma \ref{partial-1}, we prove the following proposition.
\begin{proposition}
For $g\in \mathcal{S}(\mathbb{R}^3)$ ,
we have
$$\|\mathcal{H}^{\frac{m}{2}}g\|^2_{L^2(\mathbb{R}^3)}\leq \|\nabla^{m}_{\mathcal{H}_{+}}g\|^2_{L^2(\mathbb{R}^3)}.$$
\end{proposition}
\begin{proof}
In fact, by using  \eqref{H3}, one can find that,
$$\mathcal{H}=-\Delta+\frac{|v|^2}{4}=\frac{1}{2}\sum^3_{j=1}\left(A_{+,j}A_{-,j}+A_{-,j}A_{+,j}\right).$$
Then we have
\begin{align*}
\|\mathcal{H}^{\frac{1}{2}}g\|^2_{L^2(\mathbb{R}^3)}
&=\frac{1}{2}\sum^3_{j=1}\left(\|A_{+,j}g\|^2_{L^2(\mathbb{R}^3)}+\|A_{-,j}g\|^2_{L^2(\mathbb{R}^3)}\right)\\
&\leq\sum^3_{j=1}\|A_{+,j}g\|^2_{L^2(\mathbb{R}^3)}=\|\nabla_{\mathcal{H}_{+}}g\|^2_{L^2(\mathbb{R}^3)}
\end{align*}
where we use the fact $\|A_{-,j}g\|_{L^2(\mathbb{R}^3)}\leq\|A_{+,j}g\|_{L^2(\mathbb{R}^3)}$ in  Lemma \ref{partial-1}.
For $\forall m\in\mathbb{N}^+$ with $m\ge2$, we can deduce by induction that
\begin{equation}\label{f1}
\begin{split}
&\|\nabla^m_{\mathcal{H}_{+}}g\|^2_{L^2(\mathbb{R}^3)}=\sum^3_{j=1}\|A_{+,j}\nabla^{m-1}_{\mathcal{H}_{+}}g\|^2_{L^2(\mathbb{R}^3)}\\
&=\sum^3_{j=1}\|\nabla^{m-1}_{\mathcal{H}_{+}}A_{+,j}g\|^2_{L^2(\mathbb{R}^3)}\geq \sum^3_{j=1}\|\mathcal{H}^{\frac{m-1}{2}}A_{+,j}g\|^2_{L^2(\mathbb{R}^3)}.
\end{split}
\end{equation}
By using the identity
$g=\sum_{\alpha\in \mathbb{N}^3}g_{\alpha}\psi_{\alpha},$
where $g_{\alpha}=(g,\psi_{\alpha})_{L^2(\mathbb{R}^3)}$ and \eqref{H4}-\eqref{H5}, we have
 $$\mathcal{H}^{\frac{m-1}{2}}A_{+,j}g=\sum_{\alpha\in \mathbb{N}^3}g_{\alpha}\left(|\alpha|+\frac{5}{2}\right)^{\frac{m-1}{2}}\sqrt{\alpha_j+1}\psi_{\alpha+e_j}.$$
It implies that,
\begin{equation}\label{f2}
\begin{split}
&\sum^3_{j=1}\|\mathcal{H}^{\frac{m-1}{2}}A_{+,j}g\|^2_{L^2(\mathbb{R}^3)}\\
&=
\sum^3_{j=1}\sum_{\alpha\in \mathbb{N}^3}|g_{\alpha}|^2\left(|\alpha|+\frac{5}{2}\right)^{m-1}(\alpha_j+1)\\
&\geq\sum_{\alpha\in \mathbb{N}^3}|g_{\alpha}|^2\left(|\alpha|+\frac{3}{2}\right)^{m}=\|\mathcal{H}^{\frac{m}{2}}g\|^2_{L^2(\mathbb{R}^3)}.
\end{split}
\end{equation}
Substituting the result \eqref{f2} into \eqref{f1}, we conclude that
\begin{align*}
\|\mathcal{H}^{\frac{m}{2}}g\|^2_{L^2(\mathbb{R}^3)}\leq\|\nabla^m_{\mathcal{H}_{+}}g\|^2_{L^2(\mathbb{R}^3)}.
\end{align*}
\end{proof}
Then the result of Theorem \ref{trick} implies that the solution $g$ to the Cauchy problem \eqref{eq-1} enjoys the Gelfand-Shilov $S^1_1(\mathbb{R}^3)$ smoothing effect, in fact, we have proved, there exists $c_0>0$,
such that
$$
e^{c_0\tilde{t}^{\frac 12}\mathcal{H}^{\frac 12}}g(t)\in L^\infty([0, +\infty[; L^2(\mathbb{R}^3)).
$$

\bigskip
\noindent {\bf Acknowledgements.}
The first author is supported by the Fundamental
Research Funds for the Central Universities of China, South-Central University for Nationalities (No. CZT20007). The second author is supported by the NSFC (No.12031006) and the Fundamental
Research Funds for the Central Universities of China.

\end{document}